\documentclass[12pt]{amsart}

\usepackage{graphicx}
\usepackage{mathptmx}

\usepackage{amscd}
\usepackage{amsxtra}
\usepackage{a4wide}
\usepackage{latexsym}
\usepackage{amssymb}
\usepackage{amsfonts}
\usepackage{amsmath}
\usepackage{amsrefs}
\usepackage{mathrsfs}
\usepackage{upref}
\usepackage{txfonts}
\usepackage{color}

\usepackage[bookmarksnumbered, colorlinks, plainpages]{hyperref}
\hypersetup{colorlinks=true,linkcolor=red, anchorcolor=green,
citecolor=cyan, urlcolor=red, filecolor=magenta, pdftoolbar=true}

\usepackage[left=2cm, right=2cm, top=2cm, bottom=2cm]{geometry}

\allowdisplaybreaks

\newtheorem{theorem}{Theorem}[section]
\newtheorem{lemma}[theorem]{Lemma}

\newtheorem{corollary}[theorem]{Corollary}
\theoremstyle{definition}
\newtheorem{definition}[theorem]{Definition}

\newtheorem{convention}[theorem]{Convention}

\theoremstyle{remark}

\numberwithin{equation}{section}

\def\Id{\operatorname{I}}
\def\loc{\operatorname{loc}}

\def\esup{\operatornamewithlimits{ess\,sup}}
\def\einf{\operatornamewithlimits{ess\,inf}}

\def\N{\mathbb N}
\def\R{\mathbb R}

\def\ap{\approx}

\def\mf{\mathfrak M}
\def\rad{\operatorname{rad}}

\def\qq{\qquad}
\def\rn{\R^n}
\def\a{\alpha}

\def\up{\upsilon}
\def\la{\lambda}
\def\La{\Lambda}

\def\vp{\varphi}

\def\i{\infty}

\def\I{(0,\i)}

\def\rw{\rightarrow}
\def\up{\uparrow}
\def\dn{\downarrow}

\def\ls{\lesssim}
\def\gs{\gtrsim}

\def\R{\mathbb R}

\def\M{\mathfrak M}

\def\mp{{\mathfrak M}}
\def\W{{\mathcal W}}

\begin{document}

\setcounter{page}{1}

\title[Generalized fractional maximal functions in Lorentz spaces]{Generalized fractional maximal functions in Lorentz spaces}

\author[R.Ch.Mustafayev]{Rza Mustafayev}
\address{Department of Mathematics, Faculty of Science and Arts, Kirikkale,University, 71450 Yahsihan, Kirikkale, Turkey}
\email{rzamustafayev@gmail.com}

\author[N. B\.{I}LG\.{I}\c{C}L\.{I}]{NEV\.{I}N B\.{I}LG\.{I}\c{C}L\.{I}}
\address{Department of Mathematics, Faculty of Science and Arts, Kirikkale
University, 71450 Yahsihan, Kirikkale, Turkey}
\email{nevinbilgicli@gmail.com}



\begin{abstract}
In this paper  we give the complete characterization of the
boundedness of the generalized fractional maximal operator
\begin{equation*}
    M_{\phi,\Lambda^{\alpha}(b)}f(x) : = \sup_{Q \ni x} \frac{\|f \chi_Q\|_{\Lambda^{\alpha}(b)}}{\phi (|Q|)} \qquad (x \in \rn),
\end{equation*}
between the classical Lorentz spaces $\Lambda^p (v)$ and
$\Lambda^q(w)$ for appropriate functions $\phi$, where $0 < p,\,q <
\infty$, $0 < \alpha \le r < \infty$, $v,w,\,b$ are weight functions
on $(0,\infty)$  such that $0 < B(x): = \int_0^x b < \infty$, $x >
0$, $B \in \Delta_2$ and $B(t) / t^{\alpha / r}$ is
quasi-increasing.
\end{abstract}

\maketitle


\section{Introduction}\label{in}

Throughout the paper, we always denote by  $c$ or $C$ a positive
constant, which is independent of main parameters but it may vary
from line to line. However a constant with subscript such as $c_1$
does not change in different occurrences. By $a\lesssim b$,
 we mean that $a\leq \la b$, where $\la >0$ depends on
inessential parameters. If $a\lesssim b$ and $b\lesssim a$, we write
$a\approx b$ and say that $a$ and $b$ are  equivalent. Unless a
special remark is made, the differential element $dx$ is omitted
when the integrals under consideration are the Lebesgue integrals.
By a cube, we mean an open cube with sides parallel to the
coordinate axes.

Given two quasi-normed vector spaces $X$ and $Y$, we write $X=Y$ if
$X$ and $Y$ are equal in the algebraic and the topological sense
(their quasi-norms are equivalent). The symbol $X\hookrightarrow Y$
($Y \hookleftarrow X$) means that $X\subset Y$ and the natural
embedding $\Id$ of $X$ in $Y$ is continuous, that is, there exist a
constant $c > 0$ such that $\|z\|_Y \le c\|z\|_X$ for all $z\in X$.
The best constant of the embedding $X\hookrightarrow Y$ is
$\|\Id\|_{X \rw Y}$.

Let $\Omega$ be any measurable subset of $\rn$, $n\geq 1$. Let
$\mf(\Omega)$ denote the set of all measurable functions on $\Omega$
and $\mf_0 (\Omega)$ the class of functions in $\mf (\Omega)$ that
are finite a.e. The symbol $\mp^+ (\Omega)$ stands for the
collection of all $f\in\mp (\Omega)$ which are non-negative on
$\Omega$. The symbol $\mp^+ ((0,\infty);\dn)$ is used to denote the subset of
those functions from $\mf^+ (0,\infty)$ which are non-increasing on
$(0,\infty)$. Denote by $\mf^{\rad,\dn} = \mf^{\rad,\dn}(\rn)$ the set of all
measurable, non-negative, radially decreasing functions on $\rn$,
that is,
$$
\mf^{\rad,\dn} : = \{f \in \mf(\rn):\, f(x) = h(|x|),\,x \in \rn
~\mbox{with}~ h \in \mp^+ ((0,\infty);\dn)\}.
$$
The family of all weight functions (also called just
weights) on $\Omega$, that is, locally integrable non-negative
functions on $\Omega$, is given by $\W(\Omega)$. Everywhere in the paper, $u$, $v$ and $w$ are
weights.

For $p\in (0,\i]$ and $w\in \mp^+(\Omega)$, we define the functional
$\|\cdot\|_{p,w,\Omega}$ on $\mp (\Omega)$ by
\begin{equation*}
\|f\|_{p,w,\Omega} : = \left\{\begin{array}{cl}
\left(\int_{\Omega} |f(x)|^p w(x)\,dx \right)^{1/p} & \qq\mbox{if}\qq p<\i, \\
\esup_{\Omega} |f(x)|w(x) & \qq\mbox{if}\qq p=\i.
\end{array}
\right.
\end{equation*}

If, in addition, $w\in \W(\Omega)$, then the weighted Lebesgue space
$L^p(w,\Omega)$ is given by
\begin{equation*}
L^p(w,\Omega) = \{f\in \mp (\Omega):\,\, \|f\|_{p,w,\Omega} <
\infty\}
\end{equation*}
and it is equipped with the quasi-norm $\|\cdot\|_{p,w,\Omega}$.

When $w\equiv 1$ on $\Omega$, we write simply $L^p(\Omega)$ and
$\|\cdot\|_{p,\Omega}$ instead of $L^p(w,\Omega)$ and
$\|\cdot\|_{p,w,\Omega}$, respectively.

Suppose $f$ is a measurable a.e. finite function on ${\mathbb R}^n$.
Then its non-increasing rearrangement $f^*$ is given by
$$
f^* (t) = \inf \{\lambda > 0: |\{x \in {\mathbb R}^n:\, |f(x)| >
\lambda \}| \le t\}, \quad t \in (0,\infty),
$$
and let $f^{**}$ denotes the Hardy-Littlewood maximal function of
$f$, i.e.
$$
f^{**}(t) : = \frac{1}{t} \int_0^t f^* (\tau)\,d\tau, \quad t > 0.
$$
Quite many familiar function spaces can be defined using the
non-increasing rearrangement of a function. One of the most
important classes of such spaces are the so-called classical Lorentz
spaces.

Let $p \in (0,\infty)$ and $w \in {\mathcal W}(0,\infty)$. Then the classical
Lorentz spaces $\Lambda^p (w)$ and $\Gamma^p (w)$ consist of all
functions $f \in {\mathfrak M}(\rn)$ for which $\|f\|_{\Lambda^p(w)} : = \|f^*\|_{p,w,(0,\infty)}
< \infty$ and $\|f\|_{\Gamma^p(w)} : = \|f^{**}\|_{p,w,(0,\infty)} < \infty$, respectively.
For more information about the Lorentz $\Lambda$ and $\Gamma$ see
e.g. \cite{cpss} and the references therein.

A weak-type modification of the space $\Lambda^p(w)$ is defined by
(cf. \cites{carsorJFA,sor})
$$
\Lambda^{p,\infty}(w) : = \left\{ f \in {\mathfrak M}(\rn):
\|f\|_{\Lambda^{p,\infty}(w)} : = \sup_{t>0} f^*(t) \bigg(\int_0^t
w(\tau)\,d\tau\bigg)^{1/p} < \infty \right\}.
$$
One can easily see that $\Lambda^p(w) \hookrightarrow
\Lambda^{p,\infty}(w)$. Recall that classical and weak-type Lorentz spaces include many
familiar spaces (see, for instance, \cite{edop}).

The study of maximal operators is one of the most important topics
in harmonic analysis. These significant non-linear operators, whose
behavior are very informative in particular in differentiation
theory, provided the understanding and the inspiration for the
development of the general class of singular and potential operators
(see, for instance, \cites{stein1970,guz1975,GR,tor1986,
stein1993,graf2008,graf}).

The main example is the Hardy-Littlewood maximal function which is
defined  for locally integrable functions $f$ on $\rn$ by
$$
Mf(x) : = \sup_{Q \ni x}\frac{1}{|Q|} \int_Q |f(y)|\,dy = \frac{\|f
    \chi_Q\|_{1,\rn}}{\|\chi_Q\|_{1,\rn}}, ~~ x \in \rn,
$$
where the supremum is taken over all cubes $Q$ containing $x$.

On using the Herz and the Stein rearrangement inequalities
\begin{equation}\label{Herz-Stein}
c f^{**}(t) \le (Mf)^* (t) \le C f^{**}(t), ~ t \in (0,\infty),
\end{equation}
where $c$ and $C$ are positive constants depending only on $n$ (cf.
\cite[Chapter 3, Theorem 3.8]{benshap1988}), it is clear that  in
order to describe mapping properties of the Hardy-Littlewood maximal
operator between the classical Lorentz space $\La^p(v)$ and
$\La^q(w)$, one has to characterize the weights $v,\,w$ for which
the inequality
\begin{equation}\label{HardyIneqRestr*}
\bigg(  \int_0^{\infty} \bigg(  \int_0^t f^* (\tau)\,d\tau\bigg)^q
w(t)\,dt\bigg)^{1/q} \ls \bigg( \int_0^{\infty} f^*(t)^p
v(t)\,dt\bigg)^{1/p}
\end{equation}
holds. The first results on the problem $\Lambda^p(v)
\hookrightarrow \Gamma^p(v)$, $1 < p < \infty$, which is equivalent
to inequality \eqref{HardyIneqRestr*}, were obtained by Boyd
\cite{boyd} and in an explicit form by Ari{\~n}o and Muckenhoupt
\cite{arinomuck}. The problem with $w \neq v$ and $p\neq q$, $1 <
p,\,q < \infty$ was first successfully solved by Sawyer
\cite{sawyer1990}. Many articles on this topic followed, providing
the results for a wider range of parameters. In particular, much
attention was paid to inequality \eqref{HardyIneqRestr*}; see for
instance
\cites{arinomuck,sawyer1990,steptrans,step1993,carsor1993,heinstep1993,ss,Sinn,popo,gogpick2000,bengros,gogpick2007,gjop,cgmp2008,
gogstepdokl2012_1,gogstepdokl2012_2,GogStep1,GogStep,gold2011.1,gold2011.2,johstepush,LaiShanzhong,gold2001},
survey \cite{cpss}, the monographs \cites{kufpers, kufmalpers}, for
the latest development of this subject see
\cites{GogStep,GogMusIHI}, and references given there.

The fractional maximal operator, $M_{\gamma}$, $\gamma \in (0,n)$,
is defined at $f \in L_{\loc}^1(\rn)$ by
$$
(M_{\gamma} f) (x) := \sup_{Q \ni x} |Q|^{ \gamma / n - 1} \int_{Q}
|f(y)|\,dy,\quad x \in \rn.
$$
It was shown in \cite[Theorem 1.1]{ckop} that
\begin{equation}\label{frac.max op.eq.1.}
(M_{\gamma}f)^* (t) \ls \sup_{\tau > t} \tau^{\gamma / n - 1}
\int_0^{\tau} f^*(y)\,dy \ls  (M_{\gamma} \tilde{f})^* (t)
\end{equation}
for every $f \in L_{\loc}^1(\rn)$ and $t \in \I$, where $\tilde{f}
(x) : = f^* (\omega_n |x|^n)$ and $\omega_n$ is the volume of
$S^{n-1}$. Thus, in order to characterize boundedness of the
fractional maximal operator $M_{\gamma}$ between classical Lorentz
spaces $\La^p(v)$ and $\La^q (w)$ it is necessary and sufficient to
characterize the validity of the weighted inequality
\begin{equation}\label{eq.888}
\bigg( \int_0^{\infty} \bigg[\sup_{\tau > t} \tau^{\gamma / n - 1}
\int_0^{\tau} \phi(y)\,dy\bigg]^q w(t)\,dt\bigg)^{1 / q} \ls \bigg(
\int_0^{\infty} [\phi(t)]^p v(t)\,dt\bigg)^{1 / p}
\end{equation}
for all $\phi \in \mp^+ ((0,\infty);\dn)$. 
Such a characterization was obtained in \cite{ckop} for the
particular case when $1 < p \le q <\infty$ and in \cite[Theorem
2.10]{o} in the case of more general operators and for extended
range of $p$ and $q$.

Let  $s \in (0,\infty)$, $\gamma \in [0,n)$ and $\mathbb A =
(A_0,A_{\infty}) \in {\mathbb R}^2$. Denote by
$$
\ell^{\mathbb A} (t) : = (1 + |\log t|)^{A_0} \chi_{[0,1]}(t) + (1 +
|\log t|)^{A_{\infty}} \chi_{[1,\infty)}(t),  \quad (t >0).
$$
Recall that the fractional maximal operator $M_{s,\gamma,\mathbb A}$
at $f \in \mp (\rn)$ defined in \cite{edop} by
$$
(M_{s,\gamma,\mathbb A} f) (x) : = \sup_{Q \ni x}
\frac{\|f\chi_Q\|_s}{ \|\chi_Q\|_{sn / (n-\gamma),{\mathbb A}}},
\quad x \in \rn
$$
satisfies the following equivalency
$$
(M_{s,\gamma,\mathbb A} f) (x) \ap \sup_{Q \ni x}
\frac{\|f\chi_Q\|_s}{ |Q|^{(n-\gamma)/(sn)}\ell^{\mathbb A}(|Q|)},
\quad x \in \rn.
$$
Hence, if $s = 1$, $\gamma = 0$ and ${\mathbb A} = (0,0)$, then
$M_{s,\gamma,\mathbb A}$ is equivalent to the classical
Hardy-Littlewood maximal operator $M$. If $s = 1$, $\gamma \in
(0,n)$ and ${\mathbb A} = (0,0)$, then $M_{s,\gamma,\mathbb A}$ is
equivalent to the usual fractional maximal operator $M_{\gamma}$.
Moreover, if $s = 1$, $\gamma \in [0,n)$ and ${\mathbb A} \in \R^2$,
then $M_{s,\gamma,\mathbb A} $ is the fractional maximal operator
which corresponds to potentials with logarithmic smoothness treated
in \cites{OT1,OT2}. In particular, if $\gamma = 0$, then
$M_{1,\gamma,\mathbb A}$ is the maximal operator of purely
logarithmic order.

It was shown in \cite[Theorem 3.1]{edop} that if $s \in (0,\infty)$,
$\gamma \in [0,n)$ and ${\mathbb A} = (A_0,A_{\infty}) \in \R^2$
satisfy either $\gamma \in (0,n)$, or $\gamma = 0$ and $A_0 \ge 0
\ge A_{\infty}$, then there exists a constant $C > 0$ depending only
in $n,\,s,\,\gamma$ and $\mathbb A$ such that for all $f \in \mp
(\rn)$ and every $t \in (0,\infty)$
\begin{equation}\label{frac.max op.eq.4.}
(M_{s,\gamma,\mathbb A} f)^* (t) \le C  \bigg[\sup_{t \le \tau <
\infty} \tau^{\gamma / n - 1} \ell^{-s {\mathbb A}}(\tau)
\int_0^{\tau} (f^*)^s(y)\,dy \bigg]^{1/s}.
\end{equation}
Inequality \eqref{frac.max op.eq.4.} is sharp in the sense that for
every $\vp \in \mp^+ ((0,\infty);\dn)$ there exists a function $f
\in \mp (\rn)$ such that $f^* = \vp$ a.e. on $(0,\infty)$ and for
all $t \in (0,\infty)$,
\begin{equation*}\label{frac.max op.eq.5.}
(M_{s,\gamma,\mathbb A} f)^* (t) \ge c  \bigg[\sup_{\tau > t}
\tau^{\gamma / n - 1} \ell^{-s {\mathbb A}}(\tau) \int_0^{\tau}
(f^*)^s(y)\,dy \bigg]^{1/s},
\end{equation*}
where $c$ is a positive constant with again depends only on
$n,\,s,\,\gamma$ and $\mathbb A$. Consequently, the operator
$M_{s,\gamma,\mathbb A}: \Lambda^p (v) \rw \Lambda^q(w)$ is bounded
if and only if the inequality
\begin{equation}\label{eq.opic}
\bigg( \int_0^{\infty} \bigg[\sup_{\tau > t} \tau^{\gamma / n - 1}
\ell^{-s {\mathbb A}}(\tau) \int_0^{\tau} \vp (y)\,dy
\bigg]^{q/s}w(t)\,dt\bigg)^{s / q} \ls \bigg(  \int_0^{\infty}
\vp^{p/s}(t) v(t)\,dt\bigg)^{s/p}
\end{equation}
holds for all $\phi \in \mp^+ ((0,\infty);\dn)$. The complete
characterization of inequality \eqref{eq.opic} was given in \cite[p.
17 and p. 34]{edop}. Full proofs and some further extensions and
applications can be found in \cite{edop}, \cite{edop2008}.

Given $p$ and $q$, $0 < p,\,q < \infty$, let $M_{p,q}$ denote the
maximal operator associated to the Lorentz $L^{p,q}$ spaces defined
by
$$
M_{p,q} f (x) : = \sup_{Q \ni x} \frac{\|f
\chi_Q\|_{p,q}}{\|\chi_Q\|_{p,q}} = \sup_{Q \ni x} \frac{\|f
\chi_Q\|_{p,q}}{|Q|^{1 / p}},
$$
where $\|\cdot\|_{p,q}$ is the usual Lorentz norm
$$
\|f\|_{p,q} : = \bigg( \int_0^{\infty} \big[ \tau^{1 / p} f^*
(\tau)\big]^q \frac{d\tau}{\tau}\bigg)^{1 / q}.
$$
This operator was introduced by Stein in \cite{stein1981} in order
to obtain certain endpoint results in differentiation theory. The
operator $M_{p,q}$ have been also considered by other authors, for
instance see  \cite{neug1987,leckneug,basmilruiz,perez1995,ler2005}.

It was proved in \cite{basmilruiz}, with the help of interpolation,
that if $1 \le q \le p$, then
\begin{equation}\label{eq.basmilruiz}
(M_{p,q}f)^* (t) \ls \frac{1}{t^{1/p}} \bigg( \int_0^t f^*(\tau)^q \tau^{q
/ p - 1}\,d\tau\bigg)^{1 / q}.
\end{equation}
This result was extended to more general setting of maximal
operators in \cite{mastper}. Consequently, if one knows the
characterization of the weights $u,\,v,\,w$ for which the inequality
\begin{equation}\label{HardyIneqRestr}
\bigg(  \int_0^{\infty} \bigg(  \int_0^t f^*
(\tau)u(\tau)\,d\tau\bigg)^{\beta} w(t)\,dt\bigg)^{1/\beta} \ls \bigg(
\int_0^{\infty} f^*(t)^{\alpha} v(t)\,dt\bigg)^{1/ \alpha}
\end{equation}
holds, then it is possible to describe mapping properties of
$M_{p,q}$ between the classical Lorentz spaces $\La^{\alpha}(v)$ and
$\La^{\beta}(w)$ when $1 \le q \le p$.

Let $u \in \W\I \cap C\I$, $b \in \W\I$ and $B(t) : = \int_0^t b(s)\,ds$. Assume that $b$
is such that $0 < B(t) < \infty$ for every $t \in \I$. The iterated Hardy-type operator involving suprema $T_{u,b}$ is defined at $g \in \M^+ \I$ by
$$
(T_{u,b} g)(t) : = \sup_{t \le \tau < \infty} \frac{u(\tau)}{B(\tau)} \int_0^{\tau} g(y)b(y)\,dy,\qquad t \in \I.
$$

It is easy to see that the left-hand sides of inequalities
\eqref{HardyIneqRestr*}, \eqref{eq.888}, \eqref{eq.opic} and
\eqref{HardyIneqRestr} can be interpreted as a particular examples
of operators $T_{u,b}$.

Such operators have been found indispensable in the search for
optimal pairs of rearrangement-invariant norms for which a
Sobolev-type inequality holds (cf. \cite{kerp}). They constitute a
very useful tool for characterization of the associate norm of an
operator-induced norm, which naturally appears as an optimal domain
norm in a Sobolev embedding (cf. \cite{pick2000}, \cite{pick2002}).
Supremum operators are also very useful in limiting interpolation
theory as can be seen from their appearance for example in
\cite{evop}, \cite{dok}, \cite{cwikpys}, \cite{pys}.

In \cite{GogMusISI}, complete characterization for the inequality
\begin{equation}\label{Tub.thm.1.eq.1}
\|T_{u,b}f \|_{q,w,\I} \le c \| f \|_{p,v,\I}, \qq f \in
\M^{\dn}(0,\i)
\end{equation}
for $0 < q < \infty$, $0 < p < \infty$ is given (see Theorem
\ref{Tub.thm.1}).

Inequality \eqref{Tub.thm.1.eq.1} was characterized in \cite[Theorem
3.5]{gop} under additional condition
$$
\sup_{0 < t < \infty} \frac{u(t)}{B(t)} \int_0^t
\frac{b(\tau)}{u(\tau)}\,d\tau < \infty.
$$
Note that the case when $0 < p \le 1 < q < \infty$ was not
considered in \cite{gop}. It is also worth to mention that in the
case when $1 < p < \infty$, $0 < q < p < \infty$, $q \neq 1$
\cite[Theorem 3.5]{gop} contains only discrete condition. In
\cite{gogpick2007} the new reduction theorem was obtained when $0 <
p \le 1$, and this technique allowed to characterize inequality
\eqref{Tub.thm.1.eq.1} when $b \equiv 1$, and in the case when $0 <
q< p \le 1$, \cite{gogpick2007} contains only discrete condition.

In this paper we define the following generalized fractional maximal
operator $M_{\phi,\Lambda^{\alpha}(b)}$ and characterize the
boundedness of this operator between classical and weak-type Lorentz
spaces by reducing the problem to the boundedness of the operator $T_{u,b}$ in
weighted Lebesgue spaces on the cone of non-negative non-increasing functions:

Let $0 < \alpha < \infty$, $b \in \W$ and $\phi: (0,\infty)
\rightarrow (0,\infty)$. Denote by
    \begin{equation}
    M_{\phi,\Lambda^{\alpha}(b)}f(x) : = \sup_{Q \ni x} \frac{\|f \chi_Q\|_{\Lambda^{\alpha}(b)}}{\phi (|Q|)} \qquad (x \in \rn).
    \end{equation}

Note that $M_{\phi,\Lambda^{\alpha}(b)} = M_{\gamma}$, where
$M_{\gamma}$ is  the fractional maximal operator, when $\alpha = 1$,
$b \equiv 1$ and $\phi (t) = t^{1 - \gamma / n}$ $(t >0)$ with $0 <
\gamma < n$. Moreover, $M_{\phi,\Lambda^{\alpha}(b)} \ap
M_{s,\gamma,\mathbb A}$, when $\alpha = s$, $b \equiv 1$ and $\phi
(t) = t^{(n - \gamma) / (sn)} \ell^{\mathbb A} (t)$, $(t >0)$ with
$0 < \gamma < n$ and $\mathbb A = (A_0,A_{\infty}) \in {\mathbb
R}^2$. It is worth also to mention that
$M_{\phi,\Lambda^{\alpha}(b)} = M_{p,q}$, when $\alpha = q$, $b(t) =
t^{q/ p - 1}$ and $\phi (t) = t^{1 / p}$ $(t >0)$.

The paper is organized as follows. Section \ref{pre} contains some
preliminaries along with the standard ingredients used in the
proofs. The main results are stated and proved in Section \ref{s.3}.


\section{Notations and Preliminaries}\label{pre}

Let $F$ be any non-negative set function defined on the collection
of all sets of positive finite measure. Define its maximal function
by
$$
MF (x) : = \sup_{Q \ni x} F(Q),
$$
where the supremum is taken over all cubes containing $x$.

\begin{definition}\cite[Definition 1]{ler2005}
    We say that a set function $F$ is pseudo-increasing if there is a positive constant $C > 0$ such that for any finite collection of pairwise disjoint cubes $\{Q_j\}$, we have
    \begin{equation}\label{eq.pseudo-inc}
    \min_i F(Q_i) \le C F \bigg( \bigcup_i Q_i\bigg).
    \end{equation}
\end{definition}

\begin{theorem}\cite[Theorem 1]{ler2005}\label{ler2005}
    Let $F$ be a pseudo-increasing set function. Then, for any $t > 0$,
    \begin{equation}\label{eq.ler-ineq}
    (MF)^* (t) \le C \sup_{|E| > t / 3^n} F(E),
    \end{equation}
    where $C$ is the constant appearing in \eqref{eq.pseudo-inc}, and the supremum is taken over all sets $E$ of finite measure $|E| > t / 3^n$.
\end{theorem}

We will need the following elementary inequality \cite[p.
44]{benshap1988}
\begin{equation}\label{H-L.ineq}
\int_E |f(x)|\,dx \le \int_0^{|E|} f^* (\xi)\,d\xi.
\end{equation}

A function $\phi: (0,\infty) \rightarrow (0,\infty)$ is said to
satisfy the $\Delta_2$-condition, denoted $\phi \in \Delta_2$, if
for some $C > 0$
$$
\phi (2t) \le C \phi (t) \quad \mbox{for all} \quad t > 0.
$$

A function $\phi: (0,\infty) \rightarrow (0,\infty)$ is said to be
quasi-increasing (quasi-decreasing), if for some $C > 0$
$$
\phi (t_1) \le C \phi (t_2) \qquad (\phi (t_2) \le c \phi (t_1)),
$$
whenever $0 < t_1 \le t_2 < \infty$.

A function $\phi: (0,\infty) \rightarrow (0,\infty)$ is said to
satisfy the $Q_r$-condition, $0 < r < \infty$, denoted $\phi \in
Q_r$, if for some constant $C > 0$
$$
\phi \bigg(\sum_{i=1}^n t_i \bigg) \le C \bigg( \sum_{i=1}^n
\phi(t_i)^r \bigg)^{1/r},
$$
for every finite set of non-negative real numbers
$\{t_1,\ldots,t_n\}$.

It is clear that if $\omega \in Q_r$, $0 < r < \infty$ and $g:
(0,\infty) \rw (0,\infty)$ is a quasi-decreasing function, then
$\omega \cdot g \in Q_r$.

A quasi-Banach space $X$ is a complete metrizable real vector space
whose topology is given by a quasi-norm $\|\cdot\|$ satisfying the
following three conditions: $\|x\| > 0$, $x \in X$, $x \neq 0$;
$\|\lambda x \| = |\lambda| \|x\|$, $\lambda \in \R$, $x\in X$; and
$\|x_1 + x_2\| \le C (\|x_1\| + \|x_2\|)$, $x_1,\,x_2 \in X$, where
$C \ge 1$ is a constant independent of $x_1$ and $x_2$.

A quasi-Banach function space on a measure space $(\rn,dx)$ is
defined to be a quasi-Banach space $X$ which is a subspace of
$\mp_0(\rn)$ (the topological linear space of all equivalence
classes of the real Lebesgue measurable functions equipped with the
topology of convergence in measure) such that there exists $u \in X$
with $u
> 0$ a.e. and if $|f| \le |g|$ a.e., where $g \in X$ and $f \in
\mp_0(\rn)$, then $f \in X$ and $\|f\|_X \le \|g\|_X$.

A quasi-Banach function space $X$ is said to be order continuous if
for every $f \in X$ and every sequence $\{f_n\}$ such that $0 \le
f_n \le |f|$ and $f_n \downarrow 0$ a.e. it holds $\|f_n\|_X
\rightarrow 0$.

A quasi-Banach function space $X$ is said to satisfy a lower
$r$-estimate, $0 < r < \infty$, if there exists a constant $C$ such
that
$$
\bigg( \sum_{i=1}^n \|f_i \|_X^r\bigg)^{1/r} \le C \bigg\|
\sum_{i=1}^n f_i\bigg\|_X,
$$
for every finite set of functions $\{f_1,\ldots,f_n\} \subset X$
with pairwise disjoint supports (see \cite[1.f.4]{lintzaf}).

A quasi-Banach function space $(X,\|\cdot\|_X)$ of real-valued,
locally integrable, Lebesgue measurable functions on $\rn$ is said
to be a rearrangement-invariant (r.i.) space if it satisfies the
following conditions:

\begin{enumerate}
    \item

    If $g^* \le f^*$ and $f \in X$, then $g \in X$ with $\|g\|_X \le \|f\|_X$.

    \item
    If $A$ is a Lebesgue measurable set of finite measure, then $\chi_A \in X$.

    \item
    $0 \le f_n \up$, $\sup_{n\in\N} \|f_n\|_X \le M$, imply that $f = \sup_{n \in \N} f_n \in X$ and $\|f\|_X = \sup_{n\in \N} \|f_n\|_X$.
\end{enumerate}

For each r.i. space $X$ on $\rn$, a r.i. space $\bar{X}$ on
$(0,+\infty)$ is associated such that $f \in X$ if and only if $f^*
\in \bar{X}$ and $\|f\|_X = \|f^*\|_{\bar{X}}$ (see
\cite{benshap1988}).

Most of the properties of r.-i. spaces can be formulated in terms of
the fundamental function $\varphi_X$ of $X$ defined by
$$
\varphi_X(t) = \|\chi_E \|_X,
$$
where $|E| = t$. Observe that the particular choice of the set $E$
with $|E| = t$ is immaterial by the rearrangement-invariance of $X$.
The function $\varphi_X$ is quasi-concave and continuous, except
perhaps at the origin.

Let $X$ be a quasi-Banach function space on $\rn$. By $X_{\loc}$ we
denote the space of all $f \in \mp_0(\rn)$ such that $f \chi_Q \in
X$ for every cube $Q \subset \rn$.
\begin{definition}
    Suppose that $X$ is a quasi-Banach space of
    measurable functions defined on $\rn$. Given a function $\phi: (0,\infty) \rightarrow (0,\infty)$, denote for every $f \in X_{\loc}$ by
    \begin{equation}
    M_{\phi,X} f(x) : = \sup_{Q \ni x} \frac{\|f \chi_Q\|_X}{\phi (|Q|)} \qquad (x \in \rn).
    \end{equation}
\end{definition}

It is easy to see that $M_{\phi,X} f$ is a lower-semicontinuous
function.

Note that if $X$ is r.i. quasi-Banach function space on $\rn$, then
$$
M_{\vp_X,X} f(x) = \sup_{Q \ni x} \frac{\|f\chi_Q\|_X}{\|\chi_Q\|_X}
\qq (x \in \rn).
$$

It was shown in \cite[Theorem 1]{basmilruiz}, in particular, that if
$X$ is a r.i.quasi-Banach function space satisfying a lower
$\vp_X$-estimate, that is, if there exists $C>0$ such that for all
$m \in \N$ and $\{f_i\}_{i=1}^m \subset X$ with pairwise disjoint
supports we have
$$
\vp_X \bigg( \sum_{i=1}^n \vp_X^{-1} \big( \|f_i\|_X\big)\bigg) \le
C \big\| \sum_{i=1}^n f_i\big\|_X,
$$
then there exists $C > 0$ such that for all $f \in X$ the inequality
$$
\sup_{t > 0} \vp_X(t) (M_{\vp_X,X})^* (t) \le C \|f\|_X.
$$

It was proved in \cite[Theorem 3.5]{cieskam} that if $X$ is a r.i.
order continuous quasi-Banach function space satisfying a lower
$\vp_X$-estimate, then $X$ has the Lebesgue differentiation
property, that is,
$$
\lim_{r \rw 0} \frac{\|(f - f(x)) \chi_{B(x,r)}\|_X}{\|
\chi_{B(x,r)}\|_X} = 0
$$
for all $f \in X$ and for a.a. $x \in \rn$.

Denote by
$$
V(x) : = \int_0^x v(t)\,dt  ~ \mbox{and} ~  W(x) : = \int_0^x
w(t)\,dt ~ \mbox{for all} ~ x > 0.
$$

Suppose $0 < p < \infty$ and let $w$ be a weight on $(0,\infty)$
such that $W \in \Delta_2$ and $W(\infty) = \infty$. Then the
classical Lorentz space $\La^p(w)$ is a r.i. order-continuous
quasi-Banach function space (see, for instance, \cite[Section
2.2]{carrapsor} and \cite{kammal}).

The following statement was proved in \cite{kammal}.
\begin{theorem}\cite[Theorem 7]{kammal}\label{KamMal}
    Let $w$ be a weight function such that $W \in \Delta_2$. Given  $0 < p,\,r < \infty$, the following assertions are equivalent:

    {\rm (i)} $\La^p(w)$ satisfies a lower $r$-estimate.

    {\rm (ii)} $W(t) / t^{p/r}$ is quasi-increasing and $r \ge p$.
\end{theorem}

We adopt the following conventions:
\begin{convention}\label{Notat.and.prelim.conv.1.1}
    {\rm (i)} Throughout the paper we put $0 \cdot \i = 0$, $\i / \i =
    0$ and $0/0 = 0$.

    {\rm (ii)} If $p\in [1,+\i]$, we define $p'$ by $1/p + 1/p' = 1$.

    {\rm (iii)} If $0 < q < p < \infty$, we define $r$ by $1 / r  = 1 / q - 1 / p$.
\end{convention}

Finally, for the convenience of the reader, we recall the
above-mentioned characterization of the inequality
\eqref{Tub.thm.1.eq.1}, when $0 < p,\,q < \infty$.

\begin{theorem}\cite[Theorems 5.1 and 5.5]{GogMusISI}\label{Tub.thm.1}
    Let $0 < p,\,q < \infty$ and let $u \in \W\I \cap C\I$. Assume that $b,\,v,\,w \in \W\I$
    is such that $0 < B(t) < \infty$, $0 < V(x) < \i$ and $0 < W(x) < \infty$ for all $x > 0$.
    Then inequality \eqref{Tub.thm.1.eq.1}
    is satisfied with the best constant $c$ if and only if:

    {\rm (i)} $1 < p \le q$, and in this case $c \ap A_1 + A_2$, where
    \begin{align*}
    A_1: & = \sup_{x > 0}\bigg( \bigg[\sup_{x \le \tau < \infty} \frac{u(\tau)}{B(\tau)}\bigg]^q  W(x) + \int_x^{\infty}  \bigg[\sup_{t \le \tau < \infty} \frac{u(\tau)}{B(\tau)}\bigg]^q w(t)\,dt\bigg)^{1 / q}\bigg(\int_0^x \bigg(\frac{B(y)}{V(y)}\bigg)^{p'}v(y)\,dy\bigg)^{1 / p'}, \\
    A_2: &  = \sup_{x > 0}\bigg( \bigg[\sup_{x \le \tau < \infty} \frac{u(\tau)}{V^2(\tau)}\bigg]^q  W(x) + \int_x^{\infty}  \bigg[\sup_{t \le \tau < \infty} \frac{u(\tau)}{V^2(\tau)}\bigg]^q w(t)\,dt\bigg)^{1 / q}\bigg(\int_0^x V^{p'}(y)v(y)\,dy\bigg)^{1 / p'};
    \end{align*}

    {\rm (ii)} $1 = p \le q$, and in this case $c \ap B_1 + B_2$, where
    \begin{align*}
    B_1: & = \sup_{x > 0}\bigg( \bigg[\sup_{x \le \tau < \infty} \frac{u(\tau)}{B(\tau)}\bigg]^q  W(x) + \int_x^{\infty}  \bigg[\sup_{t \le \tau < \infty} \frac{u(\tau)}{B(\tau)}\bigg]^q w(t)\,dt\bigg)^{1 / q}\bigg(\sup_{0 < y \le x} \frac{B(y)}{V(y)}\bigg), \\
    B_2: &  = \sup_{x > 0}\bigg( \bigg[\sup_{x \le \tau < \infty} \frac{u(\tau)}{V^2(\tau)}\bigg]^q  W(x) + \int_x^{\infty}  \bigg[\sup_{t \le \tau < \infty} \frac{u(\tau)}{V^2(\tau)}\bigg]^q w(t)\,dt\bigg)^{1 / q}V(x);
    \end{align*}

    {\rm (iii)} $1 < p$ and $q < p$,  and in this case $c \ap C_1 + C_2 + C_3 + C_4$, where
    \begin{align*}
    C_1: & = \bigg(\int_0^{\infty} \bigg(\int_x^{\infty}  \bigg[\sup_{t \le \tau < \infty} \frac{u(\tau)}{B(\tau)}\bigg]^q  w(t)\,dt\bigg)^{r / p} \bigg[\sup_{x \le \tau < \infty} \frac{u(\tau)}{B(\tau)}\bigg]^q \bigg(\int_0^x \bigg(\frac{B(y)}{V(y)}\bigg)^{p'}v(y)\,dy\bigg)^{r / p'} w(x)\,dx \bigg)^{1/r}, \\
    C_2: & = \bigg(\int_0^{\infty} W^{r / p}(x) \bigg[\sup_{x \le \tau < \infty} \bigg[\sup_{\tau \le y < \infty} \frac{u(y)}{B(y)}\bigg] \bigg(\int_0^x \bigg(\frac{B(y)}{V(y)}\bigg)^{p'}v(y)\,dy\bigg)^{1 / p'} \bigg]^r  w(x)\,dx \bigg)^{1/r},\\
    C_3: & = \bigg(\int_0^{\infty} \bigg(\int_x^{\infty}  \bigg[\sup_{t \le \tau < \infty} \frac{u(\tau)}{V^2(\tau)}\bigg]^q  w(t)\,dt\bigg)^{r / p} \bigg[\sup_{x \le \tau < \infty} \frac{u(\tau)}{V^2(\tau)}\bigg]^q \bigg(\int_0^x V^{p'}(y)v(y)\,dy\bigg)^{r / p'} w(x)\,dx \bigg)^{1/r}, \\
    C_4: & = \bigg(\int_0^{\infty} W^{r / p}(x) \bigg[\sup_{x \le \tau < \infty} \bigg[\sup_{\tau \le y < \infty} \frac{u(y)}{V^2(y)}\bigg] \bigg(\int_0^x V^{p'}(y)v(y)\,dy\bigg)^{1 / p'} \bigg]^r  w(x)\,dx \bigg)^{1/r};
    \end{align*}

    {\rm (iv)} $q < 1 = p$,  and in this case $c \ap D_1 + D_2 + D_3 + D_4$, where
    \begin{align*}
    D_1: & = \bigg(\int_0^{\infty} \bigg(\int_x^{\infty}  \bigg[\sup_{t \le \tau < \infty} \frac{u(\tau)}{B(\tau)}\bigg]^q  w(t)\,dt\bigg)^{r / p} \bigg[\sup_{x \le \tau < \infty} \frac{u(\tau)}{B(\tau)}\bigg]^q \bigg(\sup_{0 < y \le x} \frac{B(y)}{V(y)}\bigg)^{r} w(x)\,dx \bigg)^{1/r}, \\
    D_2: & = \bigg(\int_0^{\infty} W^{r / p}(x) \bigg[\sup_{x \le \tau < \infty} \bigg[\sup_{\tau \le y < \infty} \frac{u(y)}{B(y)}\bigg] \bigg(\sup_{0 < y \le x} \frac{B(y)}{V(y)}\bigg) \bigg]^r  w(x)\,dx \bigg)^{1/r},\\
    D_3: & = \bigg(\int_0^{\infty} \bigg(\int_x^{\infty}  \bigg[\sup_{t \le \tau < \infty} \frac{u(\tau)}{V^2(\tau)}\bigg]^q  w(t)\,dt\bigg)^{r / p} \bigg[\sup_{x \le \tau < \infty} \frac{u(\tau)}{V^2(\tau)}\bigg]^q V^r(x) w(x)\,dx \bigg)^{1/r}, \\
    D_4: & = \bigg(\int_0^{\infty} W^{r / p}(x) \bigg[\sup_{x \le \tau < \infty} \bigg[\sup_{\tau \le y < \infty} \frac{u(y)}{V^2(y)}\bigg] V(x) \bigg]^r  w(x)\,dx \bigg)^{1/r};
    \end{align*}

    {\rm (v)} $p \le q$, and in this case
    $c \ap E_1 + E_2$, where
    \begin{align*}
    E_1: & = \sup_{x > 0} \bigg( \bigg[\sup_{x \le \tau < \infty} \frac{u(\tau)}{B(\tau)}\bigg]^q \int_0^x w(t)\,dt + \int_x^{\infty} \bigg[\sup_{t \le \tau < \infty} \frac{u(\tau)}{B(\tau)}\bigg]^q w(t)\,dt\bigg)^{1 / q} \sup_{0 < y \le x} \frac{B(y)}{V^{1 / p}(y)}, \\
    E_2: & = \sup_{x > 0}\bigg( \bigg[ \sup_{x \le y < \infty} \frac{u^p(y)}{V^2(y)}\bigg]^{q / p}  \int_0^x w(t)\,dt + \int_x^{\infty}  \bigg[ \sup_{t \le y < \infty} \frac{u^p(y)}{V^2(y)}\bigg]^{q / p} w(t)\,dt\bigg)^{1 / q} V^{1 / p}(x);
    \end{align*}

    {\rm (vi)} $q < p$,  and in this case $c \ap F_1 + F_2 + F_3 + F_4$, where
    \begin{align*}
    F_1: & = \bigg(\int_0^{\infty} \bigg(\int_0^x w(t)\,dt\bigg)^{r / p} \bigg[\sup_{x \le \tau < \infty} \bigg[\sup_{\tau \le y < \infty} \frac{u(y)}{B(y)}\bigg]^p  \bigg( \sup_{0 < y \le \tau} \frac{B(y)^p}{V(y)} \bigg)\bigg]^{r / p} w(x)\,dx \bigg)^{1 / r}, \\
    F_2: & = \bigg(\int_0^{\infty} \bigg(\int_x^{\infty} \bigg[\sup_{t \le \tau < \infty} \frac{u(\tau)}{B(\tau)}\bigg]^q w(t)\,dt\bigg)^{r / p} \bigg[\sup_{0 < \tau \le x} \frac{B^p(\tau)}{V(\tau)}\bigg]^{r / p}\bigg[\sup_{x \le \tau < \infty} \frac{u(\tau)}{B(\tau)}\bigg]^q w(x)\,dx \bigg)^{1 / r}, \\
    F_3: & = \bigg( \int_0^{\infty} \bigg( \int_0^x w(t)\,dt\bigg)^{r / p} \bigg(  \sup_{x \le \tau < \infty}\bigg[\sup_{\tau \le y < \infty} \frac{u^p(y)}{V^2(y)}\bigg] V(\tau)\bigg)^{r / p} w(x)\,dx \bigg)^{1 / r}, \\
    F_4: & = \bigg(\int_0^{\infty} \bigg( \int_x^{\infty}\bigg[\sup_{t \le y < \infty} \frac{u^p(y)}{V^2(y)}\bigg]^{q / p} w(t)\,dt\bigg)^{r / p}
    \bigg[\sup_{x \le y < \infty} \frac{u^p(y)}{V^2(y)}\bigg]^{q / p} V^{r / p}(x)w(x)\,dx \bigg)^{1 / r}.
    \end{align*}
\end{theorem}

Now we give the solution of inequality \eqref{Tub.thm.1.eq.1}, when $p = \infty$ or $q = \infty$.
\begin{theorem}\label{Tub.thm.2}
    Let $0 < p < \infty$. Assume that $b \in \W\I$, $u,\,w \in \W\I \cap C\I$
    is such that $0 < B(t) < \infty$, $0 < V(x) < \i$ and $0 < W(x) < \infty$ for all $x > 0$.
    Then inequality
    \begin{equation}\label{Tub.thm.1.eq.1*}
    \|T_{u,b}f \|_{\infty,w,\I} \le c \| f \|_{p,v,\I}, \qq f \in \M^{\dn}(0,\i)
    \end{equation}
    is satisfied with the best constant $c$ if and only if:

    {\rm (i)} $1 < p$, and in this case $c \ap G_1 + G_2$, where
    \begin{align*}
    G_1: & = \sup_{x > 0}\bigg( \sup_{x \le t < \infty} \bigg[\sup_{0 < \tau \le t} w(\tau) \bigg] \frac{u(t)}{B(t)}\bigg) \bigg(\int_0^x \bigg(\frac{B(y)}{V(y)}\bigg)^{p'}v(y)\,dy\bigg)^{1 / p'}, \\
    G_2: &  = \sup_{x > 0}\bigg( \sup_{x \le t < \infty} \bigg[\sup_{0 < \tau \le t} w(\tau) \bigg] \frac{u(t)}{V^2(t)}\bigg) \bigg(\int_0^x V^{p'}(y)v(y)\,dy\bigg)^{1 / p'};
    \end{align*}

    {\rm (ii)} $p \le 1$, and in this case $c \ap H_1 + H_2$, where
    \begin{align*}
    H_1: & = \sup_{x > 0}\bigg( \sup_{0 < y \le x} \bigg( B(y) \sup_{y \le t < \infty}\bigg[\sup_{0 < \tau \le t} w(\tau) \bigg] \frac{u(t)}{B(t)} \bigg)\bigg)  V^{-1/p}(x), \\
    H_2: &  = \sup_{x > 0}\bigg( \sup_{x \le t < \infty} \bigg[\sup_{0 < \tau \le t} w(\tau) \bigg] \frac{u(t)}{B(t)}\bigg) \frac{B(x)}{V^{1 / p}(x)}.
    \end{align*}
\end{theorem}

\begin{proof}
Whenever $F,\,G$ are non-negative measurable functions on
$(0,\infty)$ and $F$ is non-increasing, then
$$
\esup_{t \in (0,\infty)} F(t)G(t) = \esup_{t \in (0,\infty)} F(t)
\esup_{\tau \in (0,t)} G(\tau);
$$
likewise, when $F$ is non-decreasing, then
$$
\esup_{t \in (0,\infty)} F(t)G(t) = \esup_{t \in (0,\infty)} F(t)
\esup_{\tau \in (t,\infty)} G(\tau).
$$
Hence
\begin{equation}\label{eq.9999}
\|T_{u,b}f \|_{\infty,w,\I} = \sup_{x>0} \bigg( \sup_{x \le t <
\infty} \bigg[ \sup_{0 < \tau \le t} w(\tau) \bigg]
\frac{u(t)}{B(t)} \bigg) \int_0^x f(y)b(y)\,dy,
\end{equation}
and inequality \eqref{Tub.thm.1.eq.1*} is equivalent to the
inequality
\begin{equation}\label{eq.7.7.}
\sup_{x>0} \bigg( \sup_{x \le t < \infty} \bigg[ \sup_{0 < \tau \le
t} w(\tau) \bigg] \frac{u(t)}{B(t)} \bigg) \int_0^x f(y)b(y)\,dy \le
c \, \bigg( \int_0^{\infty} f^p (y) v(y)\,dy\bigg)^{1/p}, \qq f \in
\M^{\dn}(0,\i).
\end{equation}

    {\rm (i)} Let $p > 1$. As in the proof of \cite[Theorem 5.1]{GogMusISI} it can be shown that \eqref{eq.7.7.} is equivalent to the following two inequalities:
\begin{align*}
\sup_{x>0} \bigg[ \sup_{0 < \tau \le x} w(\tau) \bigg] \frac{u(x)}{B(x)} \int_0^x h(y)\,dy & \le c \, \bigg( \int_0^{\infty} h^p(y) \bigg( \frac{V(y)}{B(y)} \bigg)^p v^{1-p}(y)\,dy \bigg)^{1/p},\qq h \in \M^+ (0,\i), \\
\sup_{x>0} \bigg[ \sup_{0 < \tau \le x} w(\tau) \bigg]
\frac{u(x)}{V^2(x)} \int_0^x h(y)\,dy & \le c \, \bigg(
\int_0^{\infty} \bigg(\frac{h(y)}{V(y)} \bigg)^p v^{1-p}(y)\,dy
\bigg)^{1/p},\qq h \in \M^+ (0,\i),
\end{align*}
which hold if and only if $G_1 < \infty$ and $G_2 < \infty$,
respectively (see, for instance,
\cites{opickuf,kufpers,kufmalpers}).

{\rm (ii)} Let $p \le 1$. It is known that inequality
\eqref{eq.7.7.} holds if and only if
$$
\sup_{x > 0} \bigg( \sup_{y>0} B(\min\{x,y\}) \bigg(\sup_{y \le t <
\infty} \bigg[\sup_{0 < \tau \le t} w(\tau)\bigg]
\frac{u(t)}{B(t)}\bigg)\bigg) V^{-1/p}(x)< \infty
$$
(see, for instance, \cite[Theorem 5.1, (v)]{GogMusIHI}), which is
evidently holds iff $H_1 < \infty$ and $H_2 < \infty$.

\end{proof}

\begin{theorem}\label{Tub.thm.3.1}
    Assume that $b \in \W\I$, $u,\,w \in \W\I \cap C\I$
    is such that $0 < B(t) < \infty$, $0 < V(x) < \i$ and $0 < W(x) < \infty$ for all $x > 0$.
    Then inequality
    \begin{equation}\label{Tub.thm.1.eq.1.1}
    \|T_{u,b}f \|_{\infty,w,\I} \le c \, \| f \|_{\infty,v,\I}, \qq f \in \M^{\dn}(0,\i)
    \end{equation}
    holds if and only if
    $$
    I : = \sup_{x > 0} \bigg( \int_0^x \frac{b(y)\,dy}{\esup_{\tau \in (0,y)} v(\tau)}\bigg) \bigg[ \sup_{0 < \tau \le x} w(\tau) \bigg] \frac{u(x)}{B(x)}< \infty.
    $$

    Moreover, the best constant $c$ in \eqref{Tub.thm.1.eq.1.1} satisfies $c \ap I$.
\end{theorem}

\begin{proof}
    By \eqref{eq.9999}, we know that inequality \eqref{Tub.thm.1.eq.1.1} is equivalent to the inequality
    \begin{equation}\label{eq.7.7.1}
    \sup_{x>0} \bigg( \sup_{x \le t < \infty} \bigg[ \sup_{0 < \tau \le t} w(\tau) \bigg]
    \frac{u(t)}{B(t)} \bigg) \int_0^x f(y)b(y)\,dy \le c \, \esup_{x>0} f (x) v(x), \qq f \in \M^{\dn}(0,\i),
    \end{equation}
    which holds if and only if
    $$
    \sup_{x > 0} \bigg( \int_0^x \frac{b(y)\,dy}{\esup_{\tau \in (0,y)} v(\tau)}\bigg) \bigg[ \sup_{0 < \tau \le x} w(\tau) \bigg] \frac{u(x)}{B(x)}< \infty
    $$
    (see, for instance, \cite[Theorem 5.1, (viii)]{GogMusIHI}).

\end{proof}


\section{Main results}\label{s.3}

In this section we give statements and proofs of our main results.
\begin{lemma}\label{gog}
    Let $0 < r < \infty$. Assume that $\phi \in Q_r$. Suppose that $X$ is a quasi-Banach function space  on a measure space $(\rn,dx)$. Moreover, assume that $X$  satisfy a lower $r$-estimate. Then there exists $C > 0$ such that for any function $f$ from $X$  and any finite pairwise disjoint collection cubes $\{Q_j\}$ on $\rn$
    \begin{equation}\label{main2}
    \min_i \frac{\|f\chi_{Q_i}\|_X}{\phi (|Q_i|)} \le C \,\frac{\| f\chi_{\cup_i Q_i}\|_X}{\phi (|\cup_i Q_i|)}
    \end{equation}
    holds true.
\end{lemma}

\begin{proof}
    Denote by
    $$
    A: = \min_i \frac{\|f\chi_{Q_i}\|_X}{\phi (|Q_i|)}.
    $$
    Since $\phi \in Q_r$, we have that
    $$
    A \, \phi (|\cup_i Q_i|) = A \,\phi \bigg( \sum_i |Q_i|\bigg) \ls A \,\bigg( \sum_i \phi(|Q_i|)^r\bigg)^{1/r} \le \bigg( \sum_i \|f\chi_{Q_i}\|_X^r\bigg)^{1/r}.
    $$
    On using the $r$-lower estimate property of $X$, we get that
    $$
    A \, \phi (|\cup_i Q_i|) \ls \bigg\| \sum_{i=1}^n f \chi_{Q_i}\bigg\|_X = \| f\chi_{\cup_i Q_i}\|_X.
    $$
\end{proof}

\begin{lemma}\label{main1}
    Let $0 < r < \infty$. Assume that $\phi \in Q_r$. Suppose that $X$ is a quasi-Banach function space
    satisfying a lower $r$-estimate.
    Then, for any $t > 0$,
    \begin{equation}\label{eq.ler-ineq1}
    (M_{\phi,X} f)^* (t) \le C \sup_{|E| > t / 3^n}  \frac{\|f \chi_E\|_X}{\phi (|E|)},
    \end{equation}
    where $C > 0$ is the constant appearing in \eqref{main2}.
\end{lemma}
\begin{proof}
The statement follows by Theorem \ref{ler2005} and Lemma \ref{gog}.
\end{proof}

\begin{lemma}\label{est.from.above}
    Let $0 < r < \infty$. Assume that $\phi \in Q_r$. Suppose that $X$ is a r.i. quasi-Banach function space
    satisfying a lower $r$-estimate. Then, for any $t > 0$,
    \begin{equation}\label{eq.ler-ineq1.1}
    (M_{\phi,X} f)^* (t) \le C \sup_{\tau > t}  \frac{\|f^* \chi_{[0,\tau)}\|_{\bar X}}{\phi (\tau)},
    \end{equation}
    where $C > 0$ is constant independent of $f$ and $t$.
\end{lemma}
\begin{proof}
    By Lemma \ref{main1}, we have that
    $$
    (M_{\phi,X} f)^* (t) \le C \sup_{|E| > t / 3^n}  \frac{\|f \chi_E\|_X}{\phi (|E|)} =
    C \sup_{|E| > t / 3^n}  \frac{\|(f \chi_E)^*\|_{\bar{X}}}{\phi (|E|)} \le C \sup_{|E| > t / 3^n}  \frac{\|f^*  \chi_{[0,|E|)}\|_{\bar{X}}}{\phi (|E|)} \le C \sup_{\tau > t}  \frac{\|f^* \chi_{[0,\tau)}\|_{\bar X}}{\phi (\tau)}.
    $$
\end{proof}

\begin{corollary}\label{cor.main.1}
    Let $0 < \alpha \le r < \infty$, $\phi \in Q_r$ and $b \in \W(0,\infty)$ be such that $B(\infty) = \infty$, $B \in \Delta_2$ and $B(t) / t^{\alpha / r}$ is quasi-increasing. Then there exists a constant $C > 0$ such that for any  measurable function $f$ on $\rn$ the inequality
    $$
    (M_{\phi,\Lambda^{\alpha}(b)}f)^* (t) \le C \sup_{\tau > t} \frac{\bigg(\int_0^{\tau} (f^*)^{\alpha}(y) b(y)\,dy\bigg)^{1 / \alpha}}{\phi(\tau)}
    $$
    holds.
\end{corollary}
\begin{proof}
In view of Theorem \ref{KamMal}, $\La^{\alpha}(b)$ satisfies a lower
$r$-estimate. Then the statement follows from Lemma \ref{est.from.above},
when $X = \Lambda^{\alpha}(b)$.
\end{proof}

\begin{corollary}\label{cor.basmilruiz.1}
    Let $0 < q \le p < \infty$. Then there exists a constant $C > 0$ such that for any  measurable function $f$ on $\rn$ the inequality
    \begin{equation*}
    (M_{p,q}f)^* (t) \le \frac{C}{t^{1/p}} \bigg( \int_0^t (f^*)^q (y) y^{q / p - 1}\,dy\bigg)^{1 / q}
    \end{equation*}
    holds.
\end{corollary}
\begin{proof}
    Let $\alpha = q$, $b(t) =
    t^{q/ p - 1}$ and $\phi (t) = t^{1 / p}$ $(t >0)$. Then $M_{p,q} = M_{\phi,\Lambda^{\alpha}(b)}$.
    It is clear that $B(t) \approx t^{q / p}$ $(t>0)$. Since $\phi \in Q_r$, $B \in \Delta_2$ and
    $B(t) / t^{q / r}$ is quasi-increasing when $r = p \ge q$, by Corollary \ref{cor.main.1}, we get that
    $$
    (M_{p,q}f)^* (t) \le C \sup_{\tau > t}\frac{1}{\tau^{1/p}} \bigg( \int_0^{\tau} (f^*)^q (y) y^{q / p - 1}\,dy\bigg)^{1 /
    q} = \frac{C}{t^{1/p}} \bigg( \int_0^t (f^*)^q (y) y^{q / p - 1}\,dy\bigg)^{1 /
    q}.
    $$
\end{proof}

\begin{corollary}\label{cor.edmunopic.1}
    Let  $s \in (0,\infty)$, $\gamma \in (0,n)$ and $\mathbb A = (A_0,A_{\infty}) \in {\mathbb R}^2$.
    Then there exists a constant $C > 0$ depending only in $n,\,s,\,\gamma$ and $\mathbb A$ such that for all $f \in \mp (\rn)$
    and every $t \in (0,\infty)$
    \begin{equation}\label{frac.max op.eq.4.1}
    (M_{s,\gamma,\mathbb A} f)^* (t) \le C  \bigg[\sup_{\tau > t} \tau^{\gamma / n - 1} \ell^{-s {\mathbb A}}(\tau) \int_0^{\tau} (f^*)^s(y)\,dy \bigg]^{1/s}.
    \end{equation}
\end{corollary}
\begin{proof}
    It is mentioned in the introduction that $M_{\phi,\Lambda^{\alpha}(b)} \ap M_{s,\gamma,\mathbb A}$, when $\alpha = s$,
    $b \equiv 1$ and $\phi (t) = t^{(n - \gamma) / (sn)} \ell^{\mathbb A} (t)$, $(t >0)$.
    Let $r = s$. Writing $\phi = \omega \cdot g$, where $\omega (t) = t^{1/s}$ and $g(t) = t^{- \gamma / (sn)} \ell^{\mathbb A} (t)$, $(t >0)$, observing that $\omega \in Q_s$ and $g$ is quasi-decreasing, we claim that $\phi \in Q_r$. On the other side, since $B(t) = t$, $t > 0$, we get that $B \in \Delta_2$ and $B(t) / t^{\alpha / r} \equiv 1$ is quasi-increasing. Hence, by Corollary \ref{cor.main.1}, inequality \eqref{frac.max op.eq.4.1} holds.
\end{proof}

\begin{lemma}\label{est.from.below}
    Let $0 < r < \infty$. Assume that $\phi \in \Delta_2$ is a quasi-increasing function on $(0,\infty)$. Suppose that $X$ is a r.i. quasi-Banach function space. Then, for any $t > 0$,
    \begin{equation}\label{eq.ler-ineq1.1.1}
    (M_{\phi,X} f)^* (t) \ge c \sup_{\tau > t}  \frac{\|f^* \chi_{[0,\tau)}\|_{\bar X}}{\phi (\tau)}, \qq f \in \mf^{\rad,\dn}(\rn),
    \end{equation}
    where $c > 0$ is constant independent of $f$ and $t$.
\end{lemma}
\begin{proof}
    Let $f$ be any function from $\mf^{\rad,\dn}$.
    For every $x,\,y \in \rn$ such that $|y| > |x|$, we have that
    \begin{align*}
    (M_{\phi,X}f)(x) \gs \frac{\|f \chi_{B(0,|y|)}\|_{X}}{\phi (|B(0,|y|)|)}.
    \end{align*}
    Since $(f\chi_{B(0,|y|)})^*(t) = f^*(t) \chi_{[0,|B(0,|y|)|)}(t)$, $t > 0$, we get that
    \begin{align*}
    (M_{\phi,X}f)(x) \gs  \frac{\|f^* \chi_{[0,|B(0,|y|)|)}\|_{\bar{X}}}{\phi (|B(0,|y|)|)}.
    \end{align*}
    Hence
    \begin{align*}
    (M_{\phi,X}f)(x) \gs \sup_{|y| > |x|} \frac{\|f^* \chi_{[0,|B(0,|y|)|)}\|_{\bar{X}}}{\phi (|B(0,|y|)|)} = \sup_{|y| > |x|}\frac{\|f^* \chi_{[0,\omega_n |y|^n)}\|_{\bar{X}}}{\phi(\omega_n |y|^n)} = \sup_{\tau > \omega_n |x|^n} \frac{\|f^* \chi_{[0,\tau)}\|_{\bar{X}}}{\phi(\tau)},
    \end{align*}
    where $\omega_n$ is the Lebesgue measure of the unit ball in $\rn$.

    Recall that
    $$
    f^* (t) = \sup_{|E| = t} \einf_{x \in E} |f(x)|, \quad t \in (0,\infty),
    $$
    (see, for instance, \cite[p. 33]{chongrice}).

    On taking rearrangements, we obtain that
    \begin{align*}
    (M_{\phi,X}f)^* (t) & = \sup_{|E| = t} ~ \einf_{x\in E}~ (M_{\phi,X}f)(x) \\
    & \ge \einf_{x \in B(0,(t / \omega_n)^{1 / n})} (M_{\phi,X}f)(x) \\
    & \gs \einf_{x \in B(0,(t / \omega_n)^{1 / n})} \sup_{\tau > \omega_n |x|^n} \frac{\|f^* \chi_{[0,\tau)}\|_{\bar{X}}}{\phi(\tau)} \\
    & = \einf_{0 \le s < t} \sup_{\tau > s} \frac{\|f^* \chi_{[0,\tau)}\|_{\bar{X}}}{\phi(\tau)} \\
    & = \sup_{\tau > t} \frac{\|f^* \chi_{[0,\tau)}\|_{\bar{X}}}{\phi(\tau)}.
    \end{align*}
\end{proof}

Combining Lemmas  \ref{est.from.above} and \ref{est.from.below}, we
get the following statement.
\begin{theorem}\label{main.4.1}
    Let $0 < p,q < \infty$, $0 < r < \infty$. Assume that $\phi \in Q_r$ is a quasi-increasing function on $(0,\infty)$.
    Suppose that $X$ is a r.i. quasi-Banach function space    satisfying a lower $r$-estimate. Then:

    {\rm (a)} $M_{\phi, X}$ is bounded from $\Lambda^p(v)$ to $\Lambda^q(w)$, that is, the inequality
    $$
    \|M_{\phi, X}f\|_{\Lambda^q(w)} \le C \|f\|_{\Lambda^p(v)}
    $$
    holds for all $f \in \mp (\rn)$ if and only if the inequality
    \begin{equation*}
    \bigg( \int_0^{\infty} \bigg[ \sup_{\tau > t}  \frac{\|\psi \chi_{[0,\tau)}\|_{\bar X}}{\phi (\tau)}\bigg]^q w(t)\,dt \bigg)^{1/q} \le
    C \bigg( \int_0^{\infty} (\psi(t))^p v(t)\,dt \bigg)^{1/p}
    \end{equation*}
    holds for all $\psi \in \mp^+ ((0,\infty);\dn)$.

    {\rm (b)} $M_{\phi, X}$ is bounded from $\Lambda^p(v)$ to $\Lambda^{q,\infty}(w)$, that is, the inequality
    $$
    \|M_{\phi, X}f\|_{\Lambda^{q,\infty}(w)} \le C \|f\|_{\Lambda^p(v)}
    $$
    holds for all $f \in \mp (\rn)$ if and only if the inequality
    \begin{equation*}
    \sup_{t > 0 } (W(t))^{1/q} \sup_{\tau > t}  \frac{\|\psi \chi_{[0,\tau)}\|_{\bar X}}{\phi (\tau)}  \le C \bigg( \int_0^{\infty} (\psi(t))^p v(t)\,dt \bigg)^{1/p}
    \end{equation*}
    holds for all $\psi \in \mp^+ ((0,\infty);\dn)$.

    {\rm (c)} $M_{\phi, X}$ is bounded from $\Lambda^{p,\infty}(v)$ to $\Lambda^{q,\infty}(w)$, that is, the inequality
    $$
    \|M_{\phi, X}f\|_{\Lambda^{q,\infty}(w)} \le C \|f\|_{\Lambda^{p,\infty}(v)}
    $$
    holds for all $f \in \mp (\rn)$ if and only if the inequality
    \begin{equation*}
    \sup_{t > 0 } (W(t))^{1/q} \sup_{\tau > t}  \frac{\|\psi \chi_{[0,\tau)}\|_{\bar X}}{\phi (\tau)}  \le C \sup_{t > 0 } (V(t))^{1/q} \psi(t)
    \end{equation*}
    holds for all $\psi \in \mp^+ ((0,\infty);\dn)$.
\end{theorem}


\subsection{Boundedness of $M_{\phi, \Lambda^{\alpha}(b)}: \Lambda^p(v) \rightarrow \Lambda^q(w)$, $0 < p,\,q < \infty$}

\begin{theorem}\label{main}
    Let $0 < p,q < \infty$, $0 < \alpha \le r < \infty$ and $v,\,w \in \W\I$. Assume that $\phi \in Q_r$ is a quasi-increasing function. Moreover, assume that $b \in \W\I$ is such that $0 < B(t) < \infty$ for all $x > 0$, $B(\infty) = \infty$, $B \in \Delta_2$ and $B(t) / t^{\alpha / r}$ is quasi-increasing. Then $M_{\phi, \Lambda^{\alpha}(b)}$ is bounded from $\Lambda^p(v)$ to $\Lambda^q(w)$, that is, the inequality
    $$
    \|M_{\phi, \Lambda^{\alpha}(b)}\|_{\Lambda^q(w)} \le C \|f\|_{\Lambda^p(v)}
    $$
    holds for all $f \in \mp (\rn)$ if and only if the inequality
    \begin{equation}\label{Tub.thm.1.eq.2}
    \|T_{B / \phi^{\alpha},b} \psi \|_{q / {\alpha},w,\I} \le C^{\alpha} \| \psi \|_{p / {\alpha},v,\I}
    \end{equation}
    holds for all $\psi \in \mp^+ ((0,\infty);\dn)$.
\end{theorem}

\begin{proof}
The statement follows from Theorem \ref{main.4.1}, (a), when $X =
\Lambda^{\alpha}(b)$.
\end{proof}

\begin{theorem}
    Let $0 < p,q < \infty$, $0 < \alpha \le r < \infty$ and $v,\,w \in \W\I$. Assume that $\phi \in Q_r$ is a quasi-increasing function. Moreover, assume that $b \in \W\I$ is such that $0 < B(t) < \infty$ for all $x > 0$, $B(\infty) = \infty$, $B \in \Delta_2$ and $B(t) / t^{\alpha / r}$ is quasi-increasing. Then $M_{\phi, \Lambda^{\alpha}(b)}$ is bounded from $\Lambda^p(v)$ to $\Lambda^q(w)$ if and only if:

    {\rm (i)} $\alpha < p \le q$, and in this case $c \ap {\mathcal A}_1 + {\mathcal A}_2$, where
    \begin{align*}
    {\mathcal A}_1: & = \sup_{x > 0}\bigg(  \phi^{-q}(x)  W(x) + \int_x^{\infty}   \phi^{-q}(t) w(t)\,dt\bigg)^{\frac{1}{q}}\bigg(\int_0^x \bigg(\frac{B(y)}{V(y)}\bigg)^{\frac{p}{p - \alpha}}v(y)\,dy\bigg)^{\frac{p - \alpha}{p\alpha}}, \\
    {\mathcal A}_2: &  = \sup_{x > 0}\bigg( \bigg[\sup_{x \le \tau < \infty} \frac{B(\tau)}{\phi^{\alpha}(\tau)V^2(\tau)}\bigg]^{\frac{q}{\alpha}}  W(x) + \int_x^{\infty}  \bigg[\sup_{t \le \tau < \infty} \frac{B(\tau)}{\phi^{\alpha}(\tau)V^2(\tau)}\bigg]^{\frac{q}{\alpha}} w(t)\,dt\bigg)^{\frac{1}{q}}\bigg(\int_0^x V^{\frac{p}{p -\alpha}}v\bigg)^{\frac{p-\alpha}{p\alpha}};
    \end{align*}

    {\rm (ii)} $\alpha = p \le q$, and in this case $c \ap {\mathcal B}_1 + {\mathcal B}_2$, where
    \begin{align*}
    {\mathcal B}_1: & = \sup_{x > 0}\bigg(  \phi^{-q}(x) W(x) + \int_x^{\infty}   \phi^{-q}(t) w(t)\,dt\bigg)^{1 / q}\bigg(\sup_{0 < y \le x} \frac{B(y)}{V(y)}\bigg)^{\frac{1}{\a}}, \\
    {\mathcal B}_2: &  = \sup_{x > 0}\bigg( \bigg[\sup_{x \le \tau < \infty} \frac{B(\tau)}{\phi^{\alpha}(\tau)V^2(\tau)}\bigg]^{\frac{q}{\alpha}}  W(x) + \int_x^{\infty}  \bigg[\sup_{t \le \tau < \infty} \frac{B(\tau)}{\phi^{\alpha}(\tau)V^2(\tau)}\bigg]^{\frac{q}{\alpha}} w(t)\,dt\bigg)^{\frac{1}{q}}V^{\frac{1}{\a}}(x);
    \end{align*}

    {\rm (iii)} $\alpha < p$ and $q < p$,  and in this case $c \ap {\mathcal C}_1 + {\mathcal C}_2 + {\mathcal C}_3 + {\mathcal C}_4$, where
    \begin{align*}
    {\mathcal C}_1: & = \bigg(\int_0^{\infty} \bigg(\int_x^{\infty}   \phi^{-q}(t) w(t)\,dt\bigg)^{\frac{q}{p-q}}  \phi^{-q}(x) \bigg(\int_0^x \bigg(\frac{B(y)}{V(y)}\bigg)^{\frac{p}{p-\alpha}}v(y)\,dy\bigg)^{\frac{q (p - \alpha)}{\alpha (p - q)}} w(x)\,dx \bigg)^{\frac{p-q}{pq}}, \\
    {\mathcal C}_2: & = \bigg(\int_0^{\infty} W^{\frac{q}{p-q}}(x) \bigg[\sup_{x \le \tau < \infty}  \phi^{-\a}(\tau) \bigg(\int_0^{\tau} \bigg(\frac{B(y)}{V(y)}\bigg)^{\frac{p}{p-\alpha}}v(y)\,dy\bigg)^{\frac{p - \alpha}{p}} \bigg]^{\frac{pq}{\a(p-q)}}  w(x)\,dx \bigg)^{\frac{p-q}{pq}},\\
    {\mathcal C}_3: & = \bigg(\int_0^{\infty} \bigg(\int_x^{\infty}  \bigg[\sup_{t \le \tau < \infty} \frac{B(\tau)}{\phi^{\alpha}(\tau)V^2(\tau)}\bigg]^{\frac{q}{\alpha}}  w(t)\,dt\bigg)^{\frac{q}{p-q}} \bigg[\sup_{x \le \tau < \infty} \frac{B(\tau)}{\phi^{\alpha}(\tau)V^2(\tau)}\bigg]^{\frac{q}{\alpha}} \bigg(\int_0^x V^{\frac{p}{p-\alpha}}v\bigg)^{\frac{q(p-\alpha)}{\alpha (p-q)}} w(x)\,dx \bigg)^{\frac{p-q}{pq}}, \\
    {\mathcal C}_4: & = \bigg(\int_0^{\infty} W^{\frac{q}{p-q}}(x) \bigg[\sup_{x \le \tau < \infty} \bigg[\sup_{\tau \le y < \infty} \frac{B(y)}{\phi^{\alpha}(y)V^2(y)}\bigg] \bigg(\int_0^{\tau} V^{\frac{p}{p-\alpha}}v\bigg)^{\frac{p-\alpha}{p}} \bigg]^{\frac{pq}{\alpha (p-q)}}  w(x)\,dx \bigg)^{\frac{p-q}{pq}};
    \end{align*}

    {\rm (iv)} $q < \alpha = p$,  and in this case $c \ap {\mathcal D}_1 + {\mathcal D}_2 + {\mathcal D}_3 + {\mathcal D}_4$, where
    \begin{align*}
    {\mathcal D}_1: & = \bigg(\int_0^{\infty} \bigg(\int_x^{\infty}   \phi^{-q}(t) w(t)\,dt\bigg)^{\frac{q}{p-q}}  \phi^{-q}(x)\, \bigg(\sup_{0 < y \le x} \frac{B(y)}{V(y)}\bigg)^{\frac{pq}{\alpha (p-q)}} w(x)\,dx \bigg)^{\frac{p-q}{pq}}, \\
    {\mathcal D}_2: & = \bigg(\int_0^{\infty} W^{\frac{q}{p-q}}(x) \bigg[\sup_{x \le \tau < \infty}  \phi^{-\a}(\tau) \bigg(\sup_{0 < y \le \tau} \frac{B(y)}{V(y)}\bigg) \bigg]^{\frac{pq}{\a (p-q)}}  w(x)\,dx \bigg)^{\frac{p-q}{pq}},\\
    {\mathcal D}_3: & = \bigg(\int_0^{\infty} \bigg(\int_x^{\infty}  \bigg[\sup_{t \le \tau < \infty} \frac{B(\tau)}{\phi^{\a}(\tau)V^2(\tau)}\bigg]^{\frac{q}{\a}}  w(t)\,dt\bigg)^{\frac{q}{p-q}} \bigg[\sup_{x \le \tau < \infty} \frac{B(\tau)}{\phi^{\a}(\tau)V^2(\tau)}\bigg]^{\frac{q}{\a}} V^{\frac{pq}{\a (p-q)}}(x) w(x)\,dx \bigg)^{\frac{p-q}{pq}}, \\
    {\mathcal D}_4: & = \bigg(\int_0^{\infty} W^{\frac{q}{p-q}}(x) \bigg[\sup_{x \le \tau < \infty} \bigg[\sup_{\tau \le y < \infty} \frac{B(y)}{\phi^{\a}(y)V^2(y)}\bigg] V(\tau) \bigg]^{\frac{pq}{\a (p-q)}}  w(x)\,dx \bigg)^{\frac{p-q}{pq}};
    \end{align*}

    {\rm (v)} $p \le \a$, $p \le q$ and in this case
    $c \ap {\mathcal E}_1 + {\mathcal E}_2$, where
    \begin{align*}
    {\mathcal E}_1: & = \sup_{x > 0} \bigg(  \phi^{-q}(x) W(x) + \int_x^{\infty}  \phi^{-q}(t) w(t)\,dt\bigg)^{\frac{1}{q}} \sup_{0 < y \le x} \frac{B^{\frac{1}{\a}}(y)}{V^{\frac{1}{p}}(y)}, \\
    {\mathcal E}_2: & = \sup_{x > 0}\bigg( \bigg[ \sup_{x \le y < \infty} \frac{B^{\frac{1}{\a}}(y)}{\phi(y)V^{\frac{2}{p}}(y)}\bigg]^{q}  W(x) + \int_x^{\infty}  \bigg[ \sup_{t \le y < \infty} \frac{B^{\frac{1}{\a}}(y)}{\phi(y)V^{\frac{2}{p}}(y)}\bigg]^{q} w(t)\,dt\bigg)^{\frac{1}{q}} V^{\frac{1}{p}}(x);
    \end{align*}

    {\rm (vi)} $p \le \a$, $q < p$,  and in this case $c \ap {\mathcal F}_1 + {\mathcal F}_2 + {\mathcal F}_3 + {\mathcal F}_4$, where
    \begin{align*}
    {\mathcal F}_1: & = \bigg(\int_0^{\infty} W^{\frac{q}{p-q}}(x) \bigg[\sup_{x \le \tau < \infty}  \phi^{-q}(\tau)  \bigg( \sup_{0 < y \le \tau} \frac{B(y)}{V^{\frac{\a}{p}}(y)} \bigg)\bigg]^{\frac{pq}{\a (p-q)}} w(x)\,dx \bigg)^{\frac{p-q}{pq}}, \\
    {\mathcal F}_2: & = \bigg(\int_0^{\infty} \bigg(\int_x^{\infty}  \phi^{-q}(t) w(t)\,dt\bigg)^{\frac{q}{p-q}} \bigg[\sup_{0 < \tau \le x} \frac{B(\tau)}{V^{\frac{\a}{p}}(\tau)}\bigg]^{\frac{pq}{\a (p-q)}} \phi^{-q}(x)w(x)\,dx \bigg)^{\frac{p-q}{pq}}, \\
    {\mathcal F}_3: & = \bigg( \int_0^{\infty} W^{\frac{q}{p-q}}(x) \bigg(  \sup_{x \le \tau < \infty}\bigg[\sup_{\tau \le y < \infty} \frac{B^{\frac{1}{\a}}(y)}{\phi (y)V^{\frac{2}{p}}(y)}\bigg] V^{\frac{1}{p}}(\tau)\bigg)^{\frac{pq}{p-q}} w(x)\,dx \bigg)^{\frac{p-q}{pq}}, \\
    {\mathcal F}_4: & = \bigg(\int_0^{\infty} \bigg( \int_x^{\infty}\bigg[\sup_{t \le y < \infty} \frac{B^{\frac{1}{\a}}(y)}{\phi (y)V^{\frac{2}{p}}(y)}\bigg]^{q} w(t)\,dt\bigg)^{\frac{q}{p-q}}
    \bigg[\sup_{x \le y < \infty} \frac{B^{\frac{1}{\a}}(y)}{\phi (y)V^{\frac{2}{p}}(y)}\bigg]^{q} V^{\frac{q}{p-q}}(x)w(x)\,dx \bigg)^{\frac{p-q}{pq}}.
    \end{align*}
\end{theorem}

\begin{proof}
    The statement follows from Theorems \ref{main} and \ref{Tub.thm.1}.
\end{proof}


\subsection{Boundedness of $M_{\phi, \Lambda^{\alpha}(b)}: \Lambda^p(v) \rightarrow \Lambda^{q,\infty}(w)$, $0 < p,\,q < \infty$}

\begin{theorem}\label{main*}
    Let $0 < p,q < \infty$, $0 < \alpha \le r < \infty$ and $v,\,w \in \W\I$. Assume that $\phi \in Q_r$ is a quasi-increasing function. Moreover, assume that $b \in \W\I$ is such that $0 < B(t) < \infty$ for all $x > 0$, $B(\infty) = \infty$, $B \in \Delta_2$ and $B(t) / t^{\alpha / r}$ is quasi-increasing. Then $M_{\phi, \Lambda^{\alpha}(b)}$ is bounded from $\Lambda^p(v)$ to $\Lambda^{q,\infty}(w)$, that is, the inequality
    $$
    \|M_{\phi, \Lambda^{\alpha}(b)}\|_{\Lambda^{q,\infty}(w)} \le C \|f\|_{\Lambda^p(v)}
    $$
    holds for all $f \in \mp (\rn)$ if and only if the inequality
    \begin{equation}\label{Tub.thm.1.eq.2*}
    \|T_{B / \phi^{\alpha},b} \psi \|_{\infty,W^{\alpha /q},\I} \le C^{\alpha} \| \psi \|_{p / {\alpha},v,\I}
    \end{equation}
    holds for all $\psi \in \mp^+ ((0,\infty);\dn)$.
\end{theorem}

\begin{proof}
The statement follows from Theorem \ref{main.4.1}, (b), when $X = \Lambda^{\alpha}(b)$.
\end{proof}

\begin{theorem}
    Let $0 < p,q < \infty$, $0 < \alpha \le r < \infty$ and $v,\,w \in \W\I$. Assume that $\phi \in Q_r$ is a quasi-increasing function. Moreover, assume that $b \in \W\I$ is such that $0 < B(t) < \infty$ for all $x > 0$, $B(\infty) = \infty$, $B \in \Delta_2$ and $B(t) / t^{\alpha / r}$ is quasi-increasing. Then $M_{\phi, \Lambda^{\alpha}(b)}$ is bounded from $\Lambda^p(v)$ to $\Lambda^{q,\infty}(w)$ if and only if:

    {\rm (i)} $\alpha \le p$, and in this case $c \ap {\mathcal G}_1 + {\mathcal G}_2$, where
    \begin{align*}
    {\mathcal G}_1: & = \sup_{x > 0}\bigg[  \sup_{x \le t < \infty} \frac{W^{\frac{1}{q}}(t)}{\phi(t)}\bigg]\bigg(\int_0^x \bigg(\frac{B(y)}{V(y)}\bigg)^{\frac{p}{p - \alpha}}v(y)\,dy\bigg)^{\frac{p - \alpha}{p\alpha}}, \\
    {\mathcal G}_2: &  = \sup_{x > 0}\bigg[  \sup_{x \le t < \infty} \frac{W^{\frac{1}{q}}(t)B^{\frac{1}{\alpha}}(t)}{\phi(t)V^{\frac{2}{\alpha}}(t)}\bigg]\bigg(\int_0^x V^{\frac{p}{p -\alpha}}v\bigg)^{\frac{p-\alpha}{p\alpha}};
    \end{align*}

    {\rm (ii)} $p < \alpha$, and in this case $c \ap {\mathcal H}_1 + {\mathcal H}_2$, where
    \begin{align*}
    {\mathcal H}_1: & = \sup_{x > 0} \bigg( \sup_{0 < y \le x} B^{\frac{1}{\alpha}} (y)\bigg[  \sup_{y \le t < \infty} \frac{W^{\frac{1}{q}}(t)}{\phi(t)}\bigg]\bigg) V^{- \frac{1}{p}}(x), \\
    {\mathcal H}_2: &  = \sup_{x > 0}\bigg[  \sup_{x \le t < \infty} \frac{W^{\frac{1}{q}}(t)}{\phi(t)}\bigg] \frac{B^{\frac{1}{\alpha}}(x)}{V^{\frac{1}{p}}(x)}.
    \end{align*}
\end{theorem}

\begin{proof}
The statement follows by Theorems \ref{main*} and \ref{Tub.thm.2}.
\end{proof}

\subsection{Boundedness of $M_{\phi, \Lambda^{\alpha}(b)}: \Lambda^{p,\infty}(v) \rightarrow \Lambda^{q,\infty}(w)$, $0 < p,\,q < \infty$}

\begin{theorem}\label{main.3}
    Let $0 < p,q < \infty$, $0 < \alpha \le r < \infty$ and $v,\,w \in \W\I$. Assume that $\phi \in Q_r$ is a quasi-increasing function. Moreover, assume that $b \in \W\I$ is such that $0 < B(t) < \infty$ for all $x > 0$, $B(\infty) = \infty$, $B \in \Delta_2$ and $B(t) / t^{\alpha / r}$ is quasi-increasing. Then $M_{\phi, \Lambda^{\alpha}(b)}$ is bounded from $\Lambda^p(v)$ to $\Lambda^{q,\infty}(w)$, that is, the inequality
    \begin{equation*}
    \|M_{\phi, \Lambda^{\alpha}(b)}\|_{\Lambda^{q,\infty}(w)} \le C \|f\|_{\Lambda^{p,\infty}(v)}
    \end{equation*}
    holds for all $f \in \mp (\rn)$ if and only if the inequality
    \begin{equation*}
    \|T_{B / \phi^{\alpha},b} \psi \|_{\infty,W^{\alpha /q},\I} \le C^{\alpha} \| \psi \|_{\infty,V^{\alpha / p},\I}
    \end{equation*}
    holds for all $\psi \in \mp^+ ((0,\infty);\dn)$.
\end{theorem}

\begin{proof}
The statement follows from Theorem \ref{main.4.1}, (c), when $X =
\Lambda^{\alpha}(b)$.
\end{proof}

\begin{theorem}
    Let $0 < p,q < \infty$, $0 < \alpha \le r < \infty$ and $v,\,w \in \W\I$. Assume that $\phi \in Q_r$ is a quasi-increasing function. Moreover, assume that $b \in \W\I$ is such that $0 < B(t) < \infty$ for all $x > 0$, $B(\infty) = \infty$, $B \in \Delta_2$ and $B(t) / t^{\alpha / r}$ is quasi-increasing. Then $M_{\phi, \Lambda^{\alpha}(b)}$ is bounded from $\Lambda^{p,\infty}(v)$ to $\Lambda^{q,\infty}(w)$ if and only if
    $$
    {\mathcal I} : = \sup_{x>0} \bigg( \int_0^x \frac{b(y)}{V^{\frac{\alpha}{p}}(y)} \,dy\bigg)^{\frac{1}{\alpha}} \frac{W^{\frac{1}{q}}(x)}{\phi(x)} < \infty.
    $$
    Moreover, the best constant $c$ in satisfies $c \ap {\mathcal I}$.
\end{theorem}

\begin{proof}
    The statement follows by Theorems \ref{main.3} and \ref{Tub.thm.3.1}.
\end{proof}



\begin{bibdiv}
    \begin{biblist}

\bib{arinomuck}{article}{
    author={Ari{\~n}o, M. A.},
    author={Muckenhoupt, B.},
    title={Maximal functions on classical Lorentz spaces and Hardy's
        inequality with weights for nonincreasing functions},
    journal={Trans. Amer. Math. Soc.},
    volume={320},
    date={1990},
    number={2},
    pages={727--735},
    issn={0002-9947},
    review={\MR{989570 (90k:42034)}},
    doi={10.2307/2001699},
}

\bib{basmilruiz}{article}{
    author={Bastero, J.},
    author={Milman, M.},
    author={Ruiz, F. J.},
    title={Rearrangement of Hardy-Littlewood maximal functions in Lorentz
        spaces},
    journal={Proc. Amer. Math. Soc.},
    volume={128},
    date={2000},
    number={1},
    pages={65--74},
    issn={0002-9939},
    review={\MR{1641637 (2000c:42020)}},
    doi={10.1090/S0002-9939-99-05128-X},
}

\bib{bengros}{article}{
    author={Bennett, G.},
    author={Grosse-Erdmann, K.- G.},
    title={Weighted Hardy inequalities for decreasing sequences and
        functions},
    journal={Math. Ann.},
    volume={334},
    date={2006},
    number={3},
    pages={489--531},
    issn={0025-5831},
    review={\MR{2207873 (2006m:26038)}},
    doi={10.1007/s00208-005-0678-7},
}


\bib{benshap1988}{book}{
    author={Bennett, C.},
    author={Sharpley, R.},
    title={Interpolation of operators},
    series={Pure and Applied Mathematics},
    volume={129},
    publisher={Academic Press, Inc., Boston, MA},
    date={1988},
    pages={xiv+469},
    isbn={0-12-088730-4},
    review={\MR{928802 (89e:46001)}},
}

\bib{boyd}{article}{
    author={Boyd, D. W.},
    title={The Hilbert transform on rearrangement-invariant spaces},
    journal={Canad. J. Math.},
    volume={19},
    date={1967},
    pages={599--616},
    issn={0008-414X},
    review={\MR{0212512 (35 \#3383)}},
}

\bib{cgmp2008}{article}{
    author={Carro, M.},
    author={Gogatishvili, A.},
    author={Martin, J.},
    author={Pick, L.},
    title={Weighted inequalities involving two Hardy operators with
        applications to embeddings of function spaces},
    journal={J. Operator Theory},
    volume={59},
    date={2008},
    number={2},
    pages={309--332},
    issn={0379-4024},
    review={\MR{2411048 (2009f:26024)}},
}

\bib{cpss}{article}{
    author={Carro, M.},
    author={Pick, L.},
    author={Soria, J.},
    author={Stepanov, V. D.},
    title={On embeddings between classical Lorentz spaces},
    journal={Math. Inequal. Appl.},
    volume={4},
    date={2001},
    number={3},
    pages={397--428},
    issn={1331-4343},
    review={\MR{1841071 (2002d:46026)}},
    doi={10.7153/mia-04-37},
}

\bib{carrapsor}{article}{
    author={Carro, M. J.},
    author={Raposo, J. A.},
    author={Soria, J.},
    title={Recent developments in the theory of Lorentz spaces and weighted
        inequalities},
    journal={Mem. Amer. Math. Soc.},
    volume={187},
    date={2007},
    number={877},
    pages={xii+128},
    issn={0065-9266},
    review={\MR{2308059 (2008b:42034)}},
    doi={10.1090/memo/0877},
}

\bib{carsorJFA}{article}{
    author={Carro, M. J.},
    author={Soria, J.},
    title={Weighted Lorentz spaces and the Hardy operator},
    journal={J. Funct. Anal.},
    volume={112},
    date={1993},
    number={2},
    pages={480--494},
    issn={0022-1236},
    review={\MR{1213148 (94f:42025)}},
    doi={10.1006/jfan.1993.1042},
}

\bib{carsor1993}{article}{
    author={Carro, M. J.},
    author={Soria, J.},
    title={Boundedness of some integral operators},
    journal={Canad. J. Math.},
    volume={45},
    date={1993},
    number={6},
    pages={1155--1166},
    issn={0008-414X},
    review={\MR{1247539 (95d:47064)}},
    doi={10.4153/CJM-1993-064-2},
}

\bib{chongrice}{book}{
    author={Chong, K. M.},
    author={Rice, N. M.},
    title={Equimeasurable rearrangements of functions},
    note={Queen's Papers in Pure and Applied Mathematics, No. 28},
    publisher={Queen's University, Kingston, Ont.},
    date={1971},
    pages={vi+177},
    review={\MR{0372140 (51 \#8357)}},
}

\bib{ckop}{article}{
    author={Cianchi, A.},
    author={Kerman, R.},
    author={Opic, B.},
    author={Pick, L.},
    title={A sharp rearrangement inequality for the fractional maximal
        operator},
    journal={Studia Math.},
    volume={138},
    date={2000},
    number={3},
    pages={277--284},
    issn={0039-3223},
    review={\MR{1758860 (2001h:42029)}},
}

\bib{cieskam}{article}{
    author={Ciesielski, M.},
    author={Kami{\'n}ska, A.},
    title={Lebesgue's differentiation theorems in R. I. quasi-Banach spaces
        and Lorentz spaces $\Gamma_{p,w}$},
    journal={J. Funct. Spaces Appl.},
    date={2012},
    pages={Art. ID 682960, 28},
    issn={2090-8997},
    review={\MR{2898471}},
    doi={10.1155/2012/682960},
}

\bib{cwikpys}{article}{
    author={Cwikel, M.},
    author={Pustylnik, E.},
    title={Weak type interpolation near ``endpoint'' spaces},
    journal={J. Funct. Anal.},
    volume={171},
    date={2000},
    number={2},
    pages={235--277},
    issn={0022-1236},
    review={\MR{1745635 (2001b:46118)}},
    doi={10.1006/jfan.1999.3502},
}

\bib{dok}{article}{
    author={Doktorskii, R. Ya.},
    title={Reiterative relations of the real interpolation method},
    language={Russian},
    journal={Dokl. Akad. Nauk SSSR},
    volume={321},
    date={1991},
    number={2},
    pages={241--245},
    issn={0002-3264},
    translation={
        journal={Soviet Math. Dokl.},
        volume={44},
        date={1992},
        number={3},
        pages={665--669},
        issn={0197-6788},
    },
    review={\MR{1153547 (93b:46143)}},
}

\bib{edop}{article}{
    author={Edmunds, D. E.},
    author={Opic, B.},
    title={Boundedness of fractional maximal operators between classical and
        weak-type Lorentz spaces},
    journal={Dissertationes Math. (Rozprawy Mat.)},
    volume={410},
    date={2002},
    pages={50},
    issn={0012-3862},
    review={\MR{1952673 (2004c:42040)}},
    doi={10.4064/dm410-0-1},
}

\bib{edop2008}{article}{
    author={Edmunds, D. E.},
    author={Opic, B.},
    title={Alternative characterisations of Lorentz-Karamata spaces},
    journal={Czechoslovak Math. J.},
    volume={58(133)},
    date={2008},
    number={2},
    pages={517--540},
    issn={0011-4642},
    review={\MR{2411107 (2009c:46044)}},
    doi={10.1007/s10587-008-0033-8},
}

\bib{evop}{article}{
    author={Evans, W. D.},
    author={Opic, B.},
    title={Real interpolation with logarithmic functors and reiteration},
    journal={Canad. J. Math.},
    volume={52},
    date={2000},
    number={5},
    pages={920--960},
    issn={0008-414X},
    review={\MR{1782334 (2001i:46115)}},
    doi={10.4153/CJM-2000-039-2},
}

\bib{GR}{book}{
    author={Garcia-Cuerva, J.},
    author={Rubio de Francia, J.L.},
    title={Weighted norm inequalities and related topics},
    series={North-Holland Mathematics Studies},
    volume={116},
    note={Notas de Matem\'atica [Mathematical Notes], 104},
    publisher={North-Holland Publishing Co.},
    place={Amsterdam},
    date={1985},
    pages={x+604},
}

\bib{gjop}{article}{
    author={Gogatishvili, A.},
    author={Johansson, M.},
    author={Okpoti, C. A.},
    author={Persson, L.-E.},
    title={Characterisation of embeddings in Lorentz spaces},
    journal={Bull. Austral. Math. Soc.},
    volume={76},
    date={2007},
    number={1},
    pages={69--92},
    issn={0004-9727},
    review={\MR{2343440 (2008j:46017)}},
    doi={10.1017/S0004972700039484},
}

 \bib{gogpick2000}{article}{
    author={Gogatishvili, A.},
    author={Pick, L.},
    title={Duality principles and reduction theorems},
    journal={Math. Inequal. Appl.},
    volume={3},
    date={2000},
    number={4},
    pages={539--558},
    issn={1331-4343},
    review={\MR{1786395 (2002c:46056)}},
 }

 \bib{gogpick2007}{article}{
    author={Gogatishvili, A.},
    author={Pick, L.},
    title={A reduction theorem for supremum operators},
    journal={J. Comput. Appl. Math.},
    volume={208},
    date={2007},
    number={1},
    pages={270--279},
    issn={0377-0427},
    review={\MR{2347749 (2009a:26013)}},
 }

 \bib{gogstepdokl2012_1}{article}{
    author={Gogatishvili, A.},
    author={Stepanov, V. D.},
    title={Integral operators on cones of monotone functions},
    language={Russian},
    journal={Dokl. Akad. Nauk},
    volume={446},
    date={2012},
    number={4},
    pages={367--370},
    issn={0869-5652},
    translation={
        journal={Dokl. Math.},
        volume={86},
        date={2012},
        number={2},
        pages={650--653},
        issn={1064-5624},
    },
    review={\MR{3053208}},
    doi={10.1134/S1064562412050158},
 }

 \bib{gogstepdokl2012_2}{article}{
    author={Gogatishvili, A.},
    author={Stepanov, V. D.},
    title={Operators are cones of monotone functions},
    language={Russian},
    journal={Dokl. Akad. Nauk},
    volume={445},
    date={2012},
    number={6},
    pages={618--621},
    issn={0869-5652},
    translation={
        journal={Dokl. Math.},
        volume={86},
        date={2012},
        number={1},
        pages={562--565},
        issn={1064-5624},
    },
    review={\MR{3050526}},
 }

 \bib{GogStep1}{article}{
    author={Gogatishvili, A.},
    author={Stepanov, V.D.},
    title={Reduction theorems for operators on the cones of monotone
        functions},
    journal={J. Math. Anal. Appl.},
    volume={405},
    date={2013},
    number={1},
    pages={156--172},
    issn={0022-247X},
    review={\MR{3053495}},
    doi={10.1016/j.jmaa.2013.03.046},
 }

\bib{GogMusPers2}{article}{
    author={Gogatishvili, A.},
    author={Mustafayev, R. Ch.},
    author={Persson, L.-E.},
    title={Some new iterated Hardy-type inequalities: the case $\theta = 1$},
    journal={J. Inequal. Appl.},
    date={2013},
    pages={29 pp.},
    issn={},
    doi={10.1186/1029-242X-2013-515},
}

\bib{GogMusIHI}{article}{
    author={Gogatishvili, A.},
    author={Mustafayev, R. Ch.},
    title={Weighted iterated Hardy-type inequalities},
    journal={Preprint},
    date={2015},
    pages={},
    issn={},
    doi={},
}

\bib{GogMusISI}{article}{
    author={Gogatishvili, A.},
    author={Mustafayev, R. Ch.},
    title={Iterated Hardy-type inequalities involving suprema},
    journal={Preprint},
    date={2015},
    pages={},
    issn={},
    doi={},
}

\bib{gop}{article}{
    author={Gogatishvili, A.},
    author={Opic, B.},
    author={Pick, L.},
    title={Weighted inequalities for Hardy-type operators involving suprema},
    journal={Collect. Math.},
    volume={57},
    date={2006},
    number={3},
    pages={227--255},
    issn={0010-0757},
    review={\MR{2264321 (2007g:26019)}},
}

\bib{gogpick2007}{article}{
    author={Gogatishvili, A.},
    author={Pick, L.},
    title={A reduction theorem for supremum operators},
    journal={J. Comput. Appl. Math.},
    volume={208},
    date={2007},
    number={1},
    pages={270--279},
    issn={0377-0427},
    review={\MR{2347749 (2009a:26013)}},
    doi={10.1016/j.cam.2006.10.048},
}

\bib{GogStep}{article}{
    author={Gogatishvili, A.},
    author={Stepanov, V. D.},
    title={Reduction theorems for weighted integral inequalities on the cone
        of monotone functions},
    language={Russian, with Russian summary},
    journal={Uspekhi Mat. Nauk},
    volume={68},
    date={2013},
    number={4(412)},
    pages={3--68},
    issn={0042-1316},
    translation={
        journal={Russian Math. Surveys},
        volume={68},
        date={2013},
        number={4},
        pages={597--664},
        issn={0036-0279},
    },
    review={\MR{3154814}},
}

 \bib{gold2001}{article}{
    author={Goldman, M. L.},
    title={Sharp estimates for the norms of Hardy-type operators on cones of
        quasimonotone functions},
    language={Russian, with Russian summary},
    journal={Tr. Mat. Inst. Steklova},
    volume={232},
    date={2001},
    number={Funkts. Prostran., Garmon. Anal., Differ. Uravn.},
    pages={115--143},
    issn={0371-9685},
    translation={
        journal={Proc. Steklov Inst. Math.},
        date={2001},
        number={1 (232)},
        pages={109--137},
        issn={0081-5438},
    },
    review={\MR{1851444 (2002m:42019)}},
 }

 \bib{gold2011.1}{article}{
    author={Goldman, M. L.},
    title={Order-sharp estimates for Hardy-type operators on the cones of
        functions with properties of monotonicity},
    journal={Eurasian Math. J.},
    volume={3},
    date={2012},
    number={2},
    pages={53--84},
    issn={2077-9879},
    review={\MR{3024120}},
 }

 \bib{gold2011.2}{article}{
    author={Goldman, M. L.},
    title={Order-sharp estimates for Hardy-type operators on cones of
        quasimonotone functions},
    journal={Eurasian Math. J.},
    volume={2},
    date={2011},
    number={3},
    pages={143--146},
    issn={2077-9879},
    review={\MR{2910846}},
 }

\bib{graf2008}{book}{
    author={Grafakos, L.},
    title={Classical Fourier analysis},
    series={Graduate Texts in Mathematics},
    volume={249},
    edition={2},
    publisher={Springer, New York},
    date={2008},
    pages={xvi+489},
    isbn={978-0-387-09431-1},
    review={\MR{2445437 (2011c:42001)}},
}

\bib{graf}{book}{
    author={Grafakos, L.},
    title={Modern Fourier analysis},
    series={Graduate Texts in Mathematics},
    volume={250},
    edition={2},
    publisher={Springer},
    place={New York},
    date={2009},
    pages={xvi+504},
    isbn={978-0-387-09433-5},
    review={\MR{2463316 (2011d:42001)}},
}

\bib{guz1975}{book}{
   author={de Guzm{\'a}n, M.},
   title={Differentiation of integrals in $R^{n}$},
   series={Lecture Notes in Mathematics, Vol. 481},
   note={With appendices by Antonio C\'ordoba, and Robert Fefferman, and two
    by Roberto Moriy\'on},
   publisher={Springer-Verlag, Berlin-New York},
   date={1975},
   pages={xii+266},
}

\bib{heinstep1993}{article}{
    author={Heinig, H. P.},
    author={Stepanov, V. D.},
    title={Weighted Hardy inequalities for increasing functions},
    journal={Canad. J. Math.},
    volume={45},
    date={1993},
    number={1},
    pages={104--116},
    issn={0008-414X},
    review={\MR{1200323 (93j:26011)}},
    doi={10.4153/CJM-1993-006-3},
}

\bib{johstepush}{article}{
        author={Johansson, M.},
        author={Stepanov, V. D.},
        author={Ushakova, E. P.},
        title={Hardy inequality with three measures on monotone functions},
        journal={Math. Inequal. Appl.},
        volume={11},
        date={2008},
        number={3},
        pages={393--413},
        issn={1331-4343},
        review={\MR{2431205 (2010d:26024)}},
        doi={10.7153/mia-11-30},
       }

\bib{kammal}{article}{
    author={Kami{\'n}ska, A.},
    author={Maligranda, L.},
    title={Order convexity and concavity of Lorentz spaces $\Lambda_{p,w},\ 0<p<\infty$},
    journal={Studia Math.},
    volume={160},
    date={2004},
    number={3},
    pages={267--286},
    issn={0039-3223},
    review={\MR{2033403 (2005e:46047)}},
    doi={10.4064/sm160-3-5},
}

\bib{kerp}{article}{
       author={Kerman, R.},
       author={Pick, L.},
       title={Optimal Sobolev imbeddings},
       journal={Forum Math.},
       volume={18},
       date={2006},
       number={4},
       pages={535--570},
       issn={0933-7741},
       review={\MR{2254384 (2007g:46052)}},
       doi={10.1515/FORUM.2006.028},
    }


\bib{kufpers}{book}{
    author={Kufner, A.},
    author={Persson, L.-E.},
    title={Weighted inequalities of Hardy type},
    publisher={World Scientific Publishing Co., Inc., River Edge, NJ},
    date={2003},
    pages={xviii+357},
    isbn={981-238-195-3},
    review={\MR{1982932 (2004c:42034)}},
    doi={10.1142/5129},
}

\bib{kufmalpers}{book}{
    author={Kufner, A.},
    author={Maligranda, L.},
    author={Persson, L.-E.},
    title={The Hardy inequality},
    note={About its history and some related results},
    publisher={Vydavatelsk\'y Servis, Plze\v n},
    date={2007},
    pages={162},
    isbn={978-80-86843-15-5},
    review={\MR{2351524 (2008j:26001)}},
}

\bib{LaiShanzhong}{article}{
    author={Lai, S.},
    title={Weighted norm inequalities for general operators on monotone
        functions},
    journal={Trans. Amer. Math. Soc.},
    volume={340},
    date={1993},
    number={2},
    pages={811--836},
    issn={0002-9947},
    review={\MR{1132877 (94b:42005)}},
    doi={10.2307/2154678},
}

\bib{leckneug}{article}{
    author={Leckband, M. A.},
    author={Neugebauer, C. J.},
    title={Weighted iterates and variants of the Hardy-Littlewood maximal
        operator},
    journal={Trans. Amer. Math. Soc.},
    volume={279},
    date={1983},
    number={1},
    pages={51--61},
    issn={0002-9947},
    review={\MR{704601 (85c:42021)}},
    doi={10.2307/1999370},
}

\bib{ler2005}{article}{
    author={Lerner, A. K.},
    title={A new approach to rearrangements of maximal operators},
    journal={Bull. London Math. Soc.},
    volume={37},
    date={2005},
    number={5},
    pages={771--777},
    issn={0024-6093},
    review={\MR{2164840 (2006d:42032)}},
    doi={10.1112/S0024609305004698},
}

\bib{lintzaf}{book}{
    author={Lindenstrauss, J.},
    author={Tzafriri, L.},
    title={Classical Banach spaces. II},
    series={Ergebnisse der Mathematik und ihrer Grenzgebiete [Results in
        Mathematics and Related Areas]},
    volume={97},
    note={Function spaces},
    publisher={Springer-Verlag, Berlin-New York},
    date={1979},
    pages={x+243},
    isbn={3-540-08888-1},
    review={\MR{540367 (81c:46001)}},
}

\bib{mastper}{article}{
    author={Masty{\l}o, M.},
    author={P{\'e}rez, C.},
    title={The Hardy-Littlewood maximal type operators between Banach
        function spaces},
    journal={Indiana Univ. Math. J.},
    volume={61},
    date={2012},
    number={3},
    pages={883--900},
    issn={0022-2518},
    review={\MR{3071689}},
    doi={10.1512/iumj.2012.61.4708},
}

\bib{neug1987}{article}{
    author={Neugebauer, C. J.},
    title={Iterations of Hardy-Littlewood maximal functions},
    journal={Proc. Amer. Math. Soc.},
    volume={101},
    date={1987},
    number={2},
    pages={272--276},
    issn={0002-9939},
    review={\MR{902540 (88k:42014)}},
    doi={10.2307/2045994},
}

\bib{o}{article}{
    author={Opic, B.},
    title={On boundedness of fractional maximal operators between classical
        Lorentz spaces},
    conference={
        title={Function spaces, differential operators and nonlinear analysis
        },
        address={Pudasj\"arvi},
        date={1999},
    },
    book={
        publisher={Acad. Sci. Czech Repub., Prague},
    },
    date={2000},
    pages={187--196},
    review={\MR{1755309 (2001g:42043)}},
}

\bib{opickuf}{book}{
    author={Opic, B.},
    author={Kufner, A.},
    title={Hardy-type inequalities},
    series={Pitman Research Notes in Mathematics Series},
    volume={219},
    publisher={Longman Scientific \& Technical, Harlow},
    date={1990},
    pages={xii+333},
    isbn={0-582-05198-3},
    review={\MR{1069756 (92b:26028)}},
}

\bib{OT1}{article}{
    author={Opic, B.},
    author={Trebels, W.},
    title={Bessel potentials with logarithmic components and Sobolev-type
        embeddings},
    language={English, with English and Russian summaries},
    journal={Anal. Math.},
    volume={26},
    date={2000},
    number={4},
    pages={299--319},
    issn={0133-3852},
    review={\MR{1805506 (2002b:46057)}},
    doi={10.1023/A:1005691512014},
}

\bib{OT2}{article}{
    author={Opic, B.},
    author={Trebels, W.},
    title={Sharp embeddings of Bessel potential spaces with logarithmic
        smoothness},
    journal={Math. Proc. Cambridge Philos. Soc.},
    volume={134},
    date={2003},
    number={2},
    pages={347--384},
    issn={0305-0041},
    review={\MR{1972143 (2004c:46057)}},
    doi={10.1017/S0305004102006321},
}

\bib{pick2000}{article}{
    author={Pick, L.},
    title={Supremum operators and optimal Sobolev inequalities},
    conference={
        title={Function spaces, differential operators and nonlinear analysis
        },
        address={Pudasj\"arvi},
        date={1999},
    },
    book={
        publisher={Acad. Sci. Czech Repub., Prague},
    },
    date={2000},
    pages={207--219},
    review={\MR{1755311 (2000m:46075)}},
}

\bib{perez1995}{article}{
    author={P{\'e}rez, C.},
    title={On sufficient conditions for the boundedness of the
        Hardy-Littlewood maximal operator between weighted $L^p$-spaces with
        different weights},
    journal={Proc. London Math. Soc. (3)},
    volume={71},
    date={1995},
    number={1},
    pages={135--157},
    issn={0024-6115},
    review={\MR{1327936 (96k:42023)}},
    doi={10.1112/plms/s3-71.1.135},
}

\bib{pick2002}{article}{
    author={Pick, L.},
    title={Optimal Sobolev embeddings---old and new},
    conference={
        title={Function spaces, interpolation theory and related topics (Lund,
            2000)},
    },
    book={
        publisher={de Gruyter, Berlin},
    },
    date={2002},
    pages={403--411},
    review={\MR{1943297 (2003j:46054)}},
}

\bib{popo}{article}{
    author={Popova, O. V.},
    title={Hardy-type inequalities on cones of monotone functions},
    language={Russian, with Russian summary},
    journal={Sibirsk. Mat. Zh.},
    volume={53},
    date={2012},
    number={1},
    pages={187--204},
    issn={0037-4474},
    translation={
        journal={Sib. Math. J.},
        volume={53},
        date={2012},
        number={1},
        pages={152--167},
        issn={0037-4466},
    },
    review={\MR{2962198}},
    doi={10.1134/S0037446612010132},
}

\bib{pys}{article}{
    author={Pustylnik, E.},
    title={Optimal interpolation in spaces of Lorentz-Zygmund type},
    journal={J. Anal. Math.},
    volume={79},
    date={1999},
    pages={113--157},
    issn={0021-7670},
    review={\MR{1749309 (2001a:46028)}},
    doi={10.1007/BF02788238},
}

\bib{sawyer1990}{article}{
    author={Sawyer, E.},
    title={Boundedness of classical operators on classical Lorentz spaces},
    journal={Studia Math.},
    volume={96},
    date={1990},
    number={2},
    pages={145--158},
    issn={0039-3223},
    review={\MR{1052631 (91d:26026)}},
}

\bib{Sinn}{article}{
    author={Sinnamon, G.},
    title={Transferring monotonicity in weighted norm inequalities},
    journal={Collect. Math.},
    volume={54},
    date={2003},
    number={2},
    pages={181--216},
    issn={0010-0757},
    review={\MR{1995140 (2004m:26031)}},
}

\bib{ss}{article}{
    author={Sinnamon, G.},
    author={Stepanov, V.D.},
    title={The weighted Hardy inequality: new proofs and the case $p=1$},
    journal={J. London Math. Soc. (2)},
    volume={54},
    date={1996},
    number={1},
    pages={89--101},
    issn={0024-6107},
    review={\MR{1395069 (97e:26021)}},
    doi={10.1112/jlms/54.1.89},
}

\bib{sor}{article}{
    author={Soria, J.},
    title={Lorentz spaces of weak-type},
    journal={Quart. J. Math. Oxford Ser. (2)},
    volume={49},
    date={1998},
    number={193},
    pages={93--103},
    issn={0033-5606},
    review={\MR{1617343 (99e:46037)}},
    doi={10.1093/qjmath/49.193.93},
}

\bib{steptrans}{article}{
    author={Stepanov, V. D.},
    title={The weighted Hardy's inequality for nonincreasing functions},
    journal={Trans. Amer. Math. Soc.},
    volume={338},
    date={1993},
    number={1},
    pages={173--186},
    issn={0002-9947},
    review={\MR{1097171 (93j:26012)}},
    doi={10.2307/2154450},
}

\bib{step1993}{article}{
    author={Stepanov, V. D.},
    title={Integral operators on the cone of monotone functions},
    journal={J. London Math. Soc. (2)},
    volume={48},
    date={1993},
    number={3},
    pages={465--487},
    issn={0024-6107},
    review={\MR{1241782 (94m:26025)}},
    doi={10.1112/jlms/s2-48.3.465},
}

\bib{stein1981}{article}{
    author={Stein, E. M.},
    title={Editor's note: the differentiability of functions in ${\bf
            R}^{n}$},
    journal={Ann. of Math. (2)},
    volume={113},
    date={1981},
    number={2},
    pages={383--385},
    issn={0003-486X},
    review={\MR{607898 (84j:35077)}},
}

\bib{stein1970}{book}{
            author={Stein, E.M.},
            title={Singular integrals and differentiability properties of functions},
            series={Princeton Mathematical Series, No. 30},
            publisher={Princeton University Press, Princeton, N.J.},
            date={1970},
            pages={xiv+290},
            review={\MR{0290095 (44 \#7280)}},
        }

\bib{stein1993}{book}{
            author={Stein, E.M.},
            title={Harmonic analysis: real-variable methods, orthogonality, and
                oscillatory integrals},
            series={Princeton Mathematical Series},
            volume={43},
            note={With the assistance of Timothy S. Murphy;
                Monographs in Harmonic Analysis, III},
            publisher={Princeton University Press, Princeton, NJ},
            date={1993},
            pages={xiv+695},
            isbn={0-691-03216-5},
            review={\MR{1232192 (95c:42002)}},
        }

\bib{tor1986}{book}{
            author={Torchinsky, A.},
            title={Real-variable methods in harmonic analysis},
            series={Pure and Applied Mathematics},
            volume={123},
            publisher={Academic Press, Inc., Orlando, FL},
            date={1986},
            pages={xii+462},
            isbn={0-12-695460-7},
            isbn={0-12-695461-5},
            review={\MR{869816 (88e:42001)}},
        }

\end{biblist}
\end{bibdiv}
\end{document}